\documentclass[a4paper, 12pt]{article} 
\usepackage[margin=2.5cm, top=2.5cm, bottom=2.5cm ]{geometry}	

\usepackage{amsmath}		
\usepackage{amsfonts}		
\usepackage{amsthm}
\usepackage[utf8]{inputenc}
\usepackage{color}
\usepackage{tikz}
\usepackage{verbatim}
\usepackage{slashed}
\usepackage[normalem]{ulem}

\usetikzlibrary{calc}
\usetikzlibrary{arrows}  
\usetikzlibrary{decorations.markings}
\usetikzlibrary{shapes,snakes}
\usepackage{hyperref} 
\hypersetup{
	bookmarks=true,
	unicode=false,
	pdftoolbar=true,
	pdfmenubar=true,
	pdffitwindow=false,
	pdfstartview={FitH},
	pdftitle={My title},
	pdfauthor={Author},
	pdfsubject={Subject},
	pdfcreator={Creator},
	pdfproducer={Producer},
	pdfkeywords={keyword1} {key2} {key3},
	pdfnewwindow=true,
	colorlinks=true,   
	linkcolor={blue!80!black}, 
	citecolor={CrimsonRed}, 
	filecolor={blue!80!black},
	urlcolor={blue},      
	hypertexnames=false, %
} 
\usepackage[figure]{hypcap}
\usepackage{amssymb}		
\usepackage{amstext}		
\usepackage{graphicx}		
\usepackage[utf8]{inputenc}	
\usepackage[T1]{fontenc}	
\usepackage{pdfpages}		
\usepackage{wrapfig}		
\usepackage{here}			
\usepackage{float}			
\usepackage{fancyhdr}
\usepackage{verbatim}		
\usepackage{wasysym}		
\usepackage{tocloft}		
\usepackage{siunitx}		
\usepackage[arrow, matrix, curve]{xy} 
\usepackage{wrapfig,tabularx,multirow,hyperref,icomma}
\usepackage{mathptmx}
\usepackage{lmodern}
\usepackage{courier}
\usepackage{paralist}
\usepackage{afterpage}
\usepackage{subcaption}
\usepackage[section]{placeins}
{\left.\begin{aligned}}
	{\end{aligned}\right\rbrace}
\usepackage{titletoc}
\usepackage{titlesec}
\graphicspath{{figures/}} 
\usepackage{etoolbox}
\usepackage{relsize,exscale}
\usepackage[titletoc]{appendix}
\usepackage[numbers,sort&compress]{natbib}

\usepackage{mwe} 

\usepackage{nameref}
\usepackage{comment}
\definecolor{RoyalBlue}{RGB}{35,32, 57}
\definecolor{lightblue}{RGB}{132,131,143}
\definecolor{CrimsonRed}{RGB}{130,0,29}

\newtheorem{definition}{Definition}[section]

\newtheorem{theorem}[definition]{Theorem}

\newtheorem{proposition}[definition]{Proposition}

\newtheorem{lemma}[definition]{Lemma}

\theoremstyle{definition}
\newtheorem{remark}[definition]{Remark}

\theoremstyle{theorem}

\theoremstyle{definition}

\newcommand{\z}{\boldsymbol{z}}
\newcommand{\x}{\boldsymbol{x}}

\newcommand{\y}{\boldsymbol{y}}

\newcommand{\n}{\boldsymbol{n}}

\newcommand{\ptl}{\partial}

\renewcommand{\S}{\mathcal{S}}

\renewcommand{\c}{\!\mathbb{c}}

\newcommand{\cc}{\boldsymbol{c}}
\newcommand{\dd}{\boldsymbol{d}}

\newcommand{\bdeta}{\boldsymbol{\eta}}
\newcommand{\xNP}{\x_{\textsc{np}}}
\usepackage{stmaryrd} 
\newcommand{\class}[1]{\left\llbracket #1 \right\rrbracket}
\newcommand{\CP}{\mathbb{C}\mathbb{P}}

\newcommand\numberthis{\addtocounter{equation}{1}\tag{\theequation}}
\numberwithin{equation}{section}

\makeatletter
\patchcmd{\chapter}{\if@openright\cleardoublepage\else\clearpage\fi}{}{}{}
\makeatother

\titlespacing*{\chapter}
{0cm}{1cm}{1cm}
{} \par

\titlespacing*{\section}
{0pt}{1cm}{1cm}
\titlespacing*{\subsection}
{0pt}{5.5ex plus 1ex minus .2ex}{4.3ex plus .2ex}

\usepackage{graphicx}
\usepackage{calc}
\usepackage{tocloft}
\usepackage[nottoc,notlof]{tocbibind}
\usepackage{emptypage}
\usepackage{hyperref}
\usepackage{bashful}

\renewcommand*{\cftfigname}{\figurename\space}

\usepackage{chngcntr}

\titleformat{\paragraph}
{\normalfont\normalsize\bfseries}{\theparagraph}{1em}{}
\titlespacing*{\paragraph}
{0pt}{3.25ex plus 1ex minus .2ex}{1.5ex plus .2ex}

\newcommand{\RCA}[1]{{\color{blue}{#1}}}

\usepackage[scr=boondoxo,bb=libus]{mathalpha} 

\newcommand{\VK}[1]{\textcolor{teal}{#1}}

\newcommand{\SO}{\mathrm{SO}}
\newcommand{\G}{\mathscr{G}}

\usepackage{authblk}

\title{Plane wave representations\\ for wave scattering on a sphere.}

\author[1]{Valentin D. Kunz}
\affil[1]{
{\footnotesize Department of Mathematics, Ohio State University, 231 W 18th Ave, Columbus, OH 43210, USA}}
\author[2]{Andrey V. Shanin}
\affil[2]{\footnotesize Department of Physics (Acoustics Division), Moscow State University, Leninskie Gory,
119992 Moscow, Russia}
\author[3]{Rapha\"el C. Assier}
\affil[3]{{\footnotesize Department of Mathematics, The University of Manchester, 
Oxford Road, Manchester, M13 9PL, UK}}
\date{\today}

\begin{document}
\maketitle

\begin{abstract}
Motivated by three-dimensional wave diffraction by conical shapes, we consider the two-dimensional problem of wave scattering on the real sphere $\S^2$ resulting from a point-source incidence on a scatterer $\mathbb{S}$. Using the recently introduced notion of plane waves on $\S^2$ as well as recent results on the Laplace--Beltrami Green's function $G$ on $\S^2$, we derive three plane-wave representations for the scattered field $u$. We then explain how $u$ can be analytically continued to the complexified sphere $\S^2_{\c}$. This allows us to formulate a complexified version of Green's third identity, and to provide a characterisation of complex sources. All our results are then used to draw strong analogies between planar and spherical scattering problems, notably through the introduction of directivities.
\end{abstract}

\section{Introduction}
This article is dedicated to the study of wave fields propagating on the surface of the real unit sphere $\S^2$. More precisely we are interested in the field resulting from a point source forcing on $\S^2$ and subject to Dirichlet boundary condition on some scatterer $\mathbb{S}$. The solution $g$ to this problem, which we will refer to as a \textit{tailored spherical Green's function} is important, as it arises frequently in physical contexts. It is particularly relevant for diffraction theory, where such tailored Green's function forms an integral part of formulae giving the diffraction coefficients to specific three-dimensional problems such as the quarter-plane problem or the problem of diffraction by cones \citep{Smyshlyaev1993,shanin2005modified,AssierPeake2012,sarigiannidis2026integral}. 

An important difference between wave fields propagating on a curved surface compared to those propagating on a plane (more traditionally considered), is that, in the associated Helmholtz equation, one should replace the Laplace operator by the Laplace--Beltrami operator. Moreover, the wavenumber of the Helmholtz equation should in this case be chosen such that it differs from a certain set of eigenvalues of the Laplace--Beltrami operator. A careful study of these eigenvalues can, for instance, be found in \citep{AssierPoonPeake2016}. Another important difference is that the space $\S^2$ is compact, while the plane is not.

Our first aim with this article is to reduce the computation of such tailored Green's function to that of a finite set of \textit{directivities}, which are functions depending solely on one complex variable. In order to achieve this, we rely on the plane wave representation formulae for the ``free-space'' spherical Green's function $G$ derived  in \citep{partAcomplexSphere}, and on Green's third identity on $\S^2$. This results in three distinct plane-wave representations for the scattered field $u=g-G$, expressed in terms of the directivities.

Our second aim is to draw analogies between wave scattering on the plane, and wave scattering on $\S^2$. This will be done by introducing a notion of ``far-field'' on $\S^2$ through analytic continuation of the wave fields from $\S^2$ (real dimension 2) to the complexified sphere $\S_{\c}^2$ (real dimension 4). In particular, as for the planar case, the directivities will be shown to be linked to the far-field behaviour of the scattered field.

The remainder of the article is organised as follows. In Section \ref{sec:pbformulationsphere}, we formulate the boundary value problem and Green's third identity on $\S^2$. In the motivational Section \ref{sec:motivation}, we summarise important facts and concepts of scattering problems in the plane, including the classic concepts of directivities and plane-wave representations.
Section \ref{sec:FreeSpaceGreen} is focused on plane-wave representations of the free-space Green's function $G$ on $\S^2$. It provides a brief summary of the results of \citep{partAcomplexSphere}, including the definition of plane waves on $\S^2$ in Section \ref{sec:planewavesdef}, a discussion of their singular behaviour in Section \ref{sec:singplanewaves}, and three plane-wave representations formulae in Section \ref{sec:planewaverepG}. Of the latter, two were derived in \citep{partAcomplexSphere}, and the third is new to this work (and proved in Appendix \ref{App:New01}). Section \ref{sec:PWNew02} is dedicated to the derivation of plane-wave representation formulae for the scattered field $u$ and to the introduction of the directivities. In Section \ref{sec:analytical-continuation}, using the usual Green's third identity and ideas of multidimensional complex analysis developed in \citep{AssierShanin2021AnalyticalCont}, we detail the procedure of analytical continuation of the scattered field $u$ onto the complexified sphere $\S^2_{\c}$. Moreover, we will prove a complexified version of Green's third identity, and introduce the notion of a \textit{complexified source}.  In Section \ref{sec:farfield}, we study how $u$ behaves as its argument tends to infinity from within $\S_{\c}^2$, and, making use of known asymptotics of the Legendre function, we show that the directivities are linked to the behaviour of $u$ at infinity. Finally, in Section \ref{sec:analogies} we will draw multiple analogies between the spherical and the planar case. We end the article with some conclusions and perspectives.

\section{Problem formulation} \label{sec:pbformulationsphere}

Let us denote by $\S^2=\{\x=(x_1,x_2,x_3) \in \mathbb{R}^3: x_1^2+x_2^2+x_3^2=1\}$ the real 2-sphere, and by~$\mathbb{S}$ an arbitrary scatterer with piecewise smooth boundary $\ptl \mathbb{S}$ on the sphere. The only restriction that we impose on the scatterer is for it to be small in some way. More precisely, we require that the scatterer is fully contained within a spherical cap $\mathcal{C}$. The cap $\mathcal{C}$ is centred at the North pole $\xNP=(0,0,1)$, and strictly smaller than the Northern hemisphere. This is illustrated in Figure \ref{fig:formulation-fig}~(left).

\begin{figure}[h!]
    \centering
    \includegraphics[width=\linewidth]{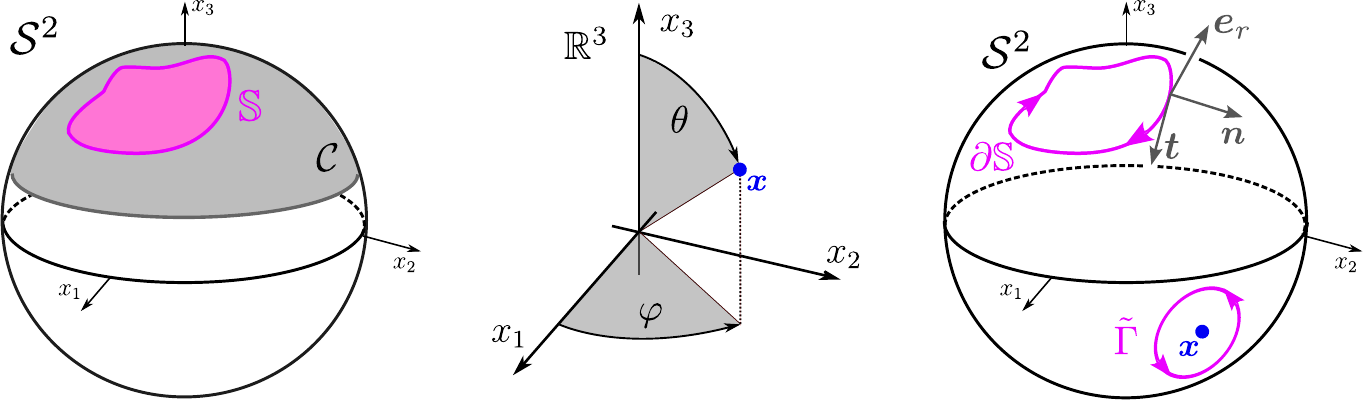}
    \caption{\small The sphere $\S^2$ and the scatterer $\mathbb{S}$ contained within the spherical cap $\mathcal{C}$ (left), the spherical coordinate system on $\S^2$ (middle), as well as two contours $\partial \mathbb{S}$ and $\tilde{\Gamma}$ which may be taken as the loop $\Gamma$ within the formulation of Green's third identity on $\mathcal{S}^2$ (right). On the right, we also see the normal vector $\n = \boldsymbol{t} \times \boldsymbol{e}_r$.}
    \label{fig:formulation-fig}
\end{figure}

Our aim is to find the function $g(\x;\x_0)$ that solves the following boundary value problem on $\S^2$.
It should obey the equation 
\begin{align}
	(\tilde{\Delta} + \lambda(\lambda+1))g= \delta(\x -\x_0), \text{ for } \x\in\S^2 \setminus \mathbb{S} \label{eq.Ch02-LB01}
\end{align}
where 
$\tilde{\Delta}$ denotes the Laplace-Beltrami operator (LBO) on $\S^2$ with respect to $\x$. The function $\delta$ is the 2D Dirac delta function on $\S^2$, the point $\x_0\in\S^2$ is a point source which can be located anywhere outside the scatterer, and $\lambda\in \mathbb{C}$ is a parameter subject to some mild restrictions to be specified later in Remark \ref{Rem.Ch02--ValueLambda}. Furthermore, we assume that $g$ is subject to Dirichlet boundary conditions on the scatterer $\mathbb{S}$:
\begin{align}
	g(\x;\x_0) =0 \text{  for  } \x\in \ptl \mathbb{S}. \label{eq.Ch02-LB03}
\end{align}
For the problem to be well-posed, one needs to ensure that Meixner condition is satisfied in the vicinity of the non-smooth points of the boundary, this condition, which can be written in integral form, effectively ensures that energy remains finite in the vicinity of these points.

The function $g$ can be thought of as the tailored Laplace--Beltrami Green's function associated with $\mathbb{S}$, or equivalently, as the \textit{total field} of the problem. It is thus natural to introduce the ``free-space'' Laplace--Beltrami Green's function $G(\x;\x_0)$, satisfying\begin{align}
	\left( 
            \tilde{\Delta} + \lambda(\lambda+1)
    \right)G = \delta(\x-\x_0), \text{  for  } \x\in\S^2, 
    \label{eq.Ch02--GreensFunctionGPDE}
\end{align}
to represent the \textit{incident field} of the problem. This function is well-known and given for $\lambda\notin\mathbb{Z}$ by
\begin{align}
	G(\x;\x_0) = \frac{1}{4 \sin(\pi \lambda)} P_{\lambda}(-\x\cdot \x_0). 
    \label{eq.Ch02--Def.GreensFunctionG}
\end{align}
Here $P_\lambda$ is the Legendre function of the first kind and of degree $\lambda$. We can now introduce the \textit{scattered field} of the problem, denoted $u(\x;\x_0)$ or simply $u(\x)$ if there is no ambiguity on the source location, and defined such that 
\begin{align}
g(\x;\x_0)=G(\x;\x_0)+u(\x;\x_0).
\end{align}
As a result, $u$ satisfies the homogeneous Laplace--Beltrami equation (LBE)
\begin{align}
	\left( 
            \tilde{\Delta} + \lambda(\lambda+1)
    \right)u(\x) = 0, \qquad \mbox{for} \quad \x \in \S^2 \setminus \mathbb{S}, 
    \label{e:HLB}
\end{align}
subject to the inhomogeneous Dirichlet boundary condition
\begin{equation}
u(\x) = - G(\x; \x_0) , \qquad \mbox{for} \quad \x \in \ptl \mathbb{S}. 
\label{e:ID}
\end{equation}

Upon introducing the usual spherical coordinates $(\theta, \varphi)\in[0,\pi]\times[0,2\pi)$, illustrated in Figure \ref{fig:formulation-fig} (right), and defined such that for $\x=(x_1,x_2,x_3)\in\S^2$ we have
\begin{align}
	x_1=\sin \theta \cos \varphi, \quad x_2=\sin \theta \sin \varphi, \quad x_3 = \cos \theta,
    \label{Ch02.eq.01}
\end{align}
it is possible to write the LBO explicitly as
\begin{align}
	\tilde{\Delta} = \partial^2_{\theta} + \frac{\cos \theta}{\sin \theta}\partial_{\theta} +
    \frac{1}{\sin^2 \theta }\partial^2_{\varphi} \, \cdot \label{Ch02.eq.002}
\end{align}

With this definition, the LBO might appear singular when $\sin(\theta)=0$, but, as discussed at length in for example \citep{partAcomplexSphere}, it is only a coordinate singularity and not a singularity of the operator. The LBO can be perfectly defined at the North and South poles through some simple rotation of the coordinate system.

Next, let us formulate \textit{Green's third identity} for the wave-field $u$, which will be one of the most important equations used throughout the remainder of this article. In order to formulate it,
consider a point $\x\in\S^2\setminus \mathbb{S}$ and a piecewise smooth, simple loop $\tilde \Gamma$ on $\S^2 \setminus \mathbb{S}$ that encircles $\x$ exactly once in the counterclockwise direction. For example, we may choose $\tilde \Gamma$ as either the oriented boundary $\partial \mathbb{S}$ of $\mathbb{S}$ shown in  Figure \ref{fig:formulation-fig}, right, or more arbitrarily, as also shown in Figure \ref{fig:formulation-fig}, right. Indeed, according to Stokes' theorem, the exact choice of contour $\tilde \Gamma$ does not matter as long as it is obtained by continuous deformation of $\partial \mathbb{S}$ without hitting $\mathbb{S}$ or $\x$. Introduce $\boldsymbol{e}_r$ to be the standard unit vector associated to spherical polar coordinates (i.e.\ $\boldsymbol{e}_r$ points outward of $\S^2$ and is normal to $\S^2$) and $\boldsymbol{t}$ to be the unit vector tangent to $\tilde \Gamma$ consistent with this contour's orientation (in particular, $\boldsymbol{t}$ is tangent to $\S^2$). We can then define the unit vector $\n$ normal to $\tilde \Gamma$ and tangent to $\S^2$ 
by $\n = \boldsymbol{t} \times \boldsymbol{e}_r $. For any function $f(\y)$ defined on $\tilde \Gamma$, one can therefore define its normal derivative $\partial_{\n}f(\y)\equiv \n\cdot\tilde\nabla f (\y)$, where $\tilde \nabla$ is the standard gradient operator restricted to $\S^2$. The Green's formula representation of $u(\x)$ is then given by 
\begin{align}
	u(\x) = \int_{\tilde\Gamma} \left[ u(\y) \partial_{\n}G(\x;\y) - G(\x;\y) \partial_{\n}u(\y)\right] d\ell, \label{e:Greens3rdSphere}
\end{align}
where $d\ell$ denotes the line element on $\tilde\Gamma$. The integration and the normal derivatives within the integrand are performed with respect to the $\y$ coordinate. This formula can be proven in a similar way to the more standard one in $\mathbb{R}^2$ as shown in \citep[Appendix A]{AssierPoonPeake2016}, for instance. 


\begin{remark}[Restriction on $\lambda$]
	\label{Rem.Ch02--ValueLambda}
	As in \cite{partAcomplexSphere}, given the formula (\ref{eq.Ch02--Def.GreensFunctionG}) it is clear that we need to have $\lambda\notin \mathbb{Z}$. These values correspond to $-\lambda(\lambda+1)$ being eigenvalues of the LBO on $\S^2$ (see e.g.\ \citep{keller1999singularities}). 
    Since in the present work we are also interested in $g$, not solely in $G$, for a similar reason, we also do not want $-\lambda(\lambda+1)$ to coincide with the eigenvalues of the LBO on $\S^2$ subjected to Dirichlet boundary condition on $\mathbb{S}$. This set corresponds again to a countable set of real values, as for example in \citep{AssierPoonPeake2016}.
	Therefore, from now on, $\lambda$ will be assumed not to belong to these two countable sets.
\end{remark}

\begin{remark}[More generality] We chose to consider the case of only one scatterer subject to Dirichlet boundary conditions for simplicity of exposition. However, everything we do in the present work applies also to arbitrarily many scatterers (as long as they are all contained within the spherical cap $\mathcal{C}$), each subject to boundary conditions that can be of the Dirichlet, Neumann or Robin type.
\end{remark}

\begin{remark}[Scatterers of specific interest] As mentioned in introduction, the spherical diffraction problems studied in this work can be related to three dimensional diffraction problems. The spherical problem typically arises through a method of separation of variables. For instance if the scatterer $\mathbb{S}$ is a spherical cap, the corresponding three dimensional problem is that of diffraction by a circular cone. If the scatterer is part of a great circle (a slit), then the corresponding three-dimensional problem is that of diffraction by a plane-sector (degenerate elliptic cone), a particular case of which is the long-standing quarter-plane problem (see, e.g.\ \cite{AssierAbrahams2021, AssierPeake2012, AssierPeake2012FarField, AssierShanin2019, AssierShanin2021VertexGreensFunctions,  assier2024contribution, lyalinov2013scattering,  sarigiannidis2026integral, Shanin2005b, shanin2005modified, shanin2012asymptotics, Smyshlyaev1993, ValyaevShanin2012}). The latter problem is of great interest to the authors and played an important motivational role in this work.  
\end{remark}

Before starting to work on the resolution of this scattering problem on the sphere, we would like to motivate our approach by first reminding the readers about important facts and notions associated with scattering problems on the plane in the next section. Our aim is to draw strong analogies between these two kinds of scattering problems that are seemingly very different due to the geometry of the ambient space (one is compact and curved, the other is flat and non-compact). These analogies will be given in Section~\ref{sec:analogies}.



 \section{Motivation: Wave scattering on a plane} \label{sec:motivation}

One of the main aims of this article is to introduce concepts for scattering on a sphere that mimic a set of concepts that have proved to be very helpful for scattering within a plane. In this section we will briefly remind the reader of some of those concepts. In order to do so, as an example, we consider the canonical diffraction problem of a plane wave scattered by a segment (or strip), as studied for instance in \cite{Williams1982,Shanin2001,Shanin2003,AssierShanin2021AnalyticalCont}. When possible we will re-use some notations from the previous section when describing similar objects in order to insist on the similarities.

Throughout this section we work with points $\x=(x_1,x_2)\in\mathbb{R}^2$ and consider the total field $u^{\text{tot}}(\x)$, the incident field $u^{\text{in}}(\x)$ and the scattered field $u(\x)$ related by
\[
u^{\text{tot}} = u^{\rm in} + u.
\]
The incident field is chosen to be a plane wave and can be written as
\[
u^{\rm in} = \exp \{ i k (x_1 \cos \left( \theta^{\rm in} \right) - x_2 \sin \left( \theta^{\rm in}) \right) \} ,
\]
for some incident angle $\theta^{\text{in}}$ defined as in Figure~\ref{f:AS01}, left.
All three wave fields satisfy the Helmholtz equation
\begin{equation}
\Delta u + k^2 u = 0, 
\label{e:AS001}
\end{equation}
where $\Delta$ is the usual two-dimensional Laplacian and the parameter $k$ represents the wavenumber.
We assume that the scatterer, represented by $\x\in[-a,a]\times \{0\}$ is subject to Dirichlet boundary conditions, meaning that $u^{\text{tot}} = 0$ on both faces of the scatterer. This problem is illustrated in \figurename~\ref{f:AS01}~(left).

To ensure that the problem is well posed, the scattered field $u$ is required to satisfy the Sommerfeld radiation condition at infinity\footnote{Note that in the spherical scattering problem of Section \ref{sec:pbformulationsphere}, no such condition is required because the ambient space is compact.}, while the total field is required to satisfy the Meixner conditions at the vertices $(\pm a, 0)$, see e.g. \cite{Babich2007}.





The following concepts (a)--(d) can be applied to this problem: 

{\paragraph*{(a) Green's identity}

It is well-known 
that the scattered field in the whole plane 
can be expressed through the field on some contour $\tilde \Gamma$ encircling the scatterer (see \figurename~\ref{f:AS01}~(centre)). 
The formula (Green's identity) is as follows: 
\begin{equation}
u(\x) = \int_{\tilde \Gamma}
\left[ 
 G(\x ;\y) \ptl_{\n} u(\y)
-
u (\y) \ptl_{\n} G(\x ;\y)
\right] d\ell.
\label{e:AS002}
\end{equation}
The observation point $\x$ should be located outside $\tilde \Gamma$ and $\y$ is understood to lie on $\tilde \Gamma$. The function $G$ is the free-space two-dimensional Helmholtz Green's function given by:
\begin{align}
G(\x;\y) = - \frac{i}{4} H_0^{(1)} (k |\x- \y|), \label{e:Greal}
\end{align}
where $H_0^{(1)}(\cdot)$ is the Hankel's function of the first kind.  The operator $\partial_{\n}$ is the standard normal derivative with respect to $\y$, where $\n$ is the usual unit normal vector to $\tilde \Gamma$ consistent with its orientation.  Note the conceptual similarity with \eqref{e:Greens3rdSphere}.



\paragraph*{(b) The directivity pattern of the scattered field.}


Consider an observation point $\x$ at a large distance from the scatterer and in the observation direction associated with an angle $\theta$, i.e.\ we can write 
\[
x_1 = r\cos (\theta) , \quad x_2 = r \sin (\theta), 
\qquad 
r \to \infty. 
\]
Using known large-argument asymptotics formulae for the Hankel's function, (\ref{e:AS002}) can be rewritten in the form: 
\begin{equation}
u(\x) \underset{r \rightarrow \infty}{=} B \frac{e^{i k r}}{\sqrt{kr}} \mathcal{D}(\theta) + \mathcal{O} \left( \frac{e^{i k r}}{(kr)^{3/2}}  \right), 
\text{ where }  
B = - \frac{e^{i \pi / 4}}{2\sqrt{2 \pi}}.
\label{e:AS007}
\end{equation}
The term $\mathcal{D}(\theta)$ is called the \textit{directivity pattern} of the scattered field and is given by 
\begin{equation}
D(\theta) =
\int_{\tilde \Gamma}
\left[ 
w(\y, \theta + \pi) \ptl_{\n} u (\y) 
-
u (\y) \ptl_{\n} w(\y, \theta + \pi ) 
\right] d\ell,
\label{e:AS008}
\end{equation}
where 
$w(\y,\theta)$ is defined for any $\y=(y_1,y_2)\in\mathbb{R}^2$ and any $\theta\in\mathbb{R}$ to be the plane wave going in the $\theta$ direction:
\begin{equation}
w(\y , \theta) = \exp\{ 
i k (y_1 \cos \theta + y_2 \sin \theta)
\}. 
\label{e:AS009}
\end{equation}

This directivity pattern is the main characteristic of the scattered field. Although it is a coefficient appearing in the leading order term of the 
far-field asymptotic expansion (\ref{e:AS007}), it can be used to reconstruct all the other terms in the expansion as will be briefly explained in the next subsection.

\paragraph*{(c) Plane-wave representation of the scattered field.}

Consider the upper-half space $\{x_2 > 0\}$. One can show (see e.g.\ \cite{Babich2007}) that the scattered field in this domain has the plane-wave representation 
\begin{equation}
u (\x) = \int_{\gamma} \mathcal{A}(\beta) \, w(\x , \beta) \, d\beta, 
\label{e:AS010}
\end{equation}
where $\mathcal{A}(\beta)$ is a spectral function directly related to the directivity pattern through 
\[
\mathcal{A}(\beta) = \frac{1}{4\pi i} \mathcal{D}(\beta),
\]
and $\gamma$ is the contour in the $\beta$-plane shown in \figurename~\ref{f:AS01}~(right). A modification of this contour, the method of sliding contours (see e.g.\ \cite{Babich2007,Nethercote2020ReviewPaper}), allows one to describe the scattered field for any $\x$, thereby lifting the upper-half space restriction. Moreover, using, for instance, the saddle point method to asymptotically write the right-hand-side of \eqref{e:AS010} as a power series in $r=|\x|$ as $r \to \infty$, we find that all terms in the expansion \eqref{e:AS007} are determined by $\mathcal{D}(\theta)$.

\begin{figure}[h!]
	\centering
	\includegraphics[width=\textwidth]{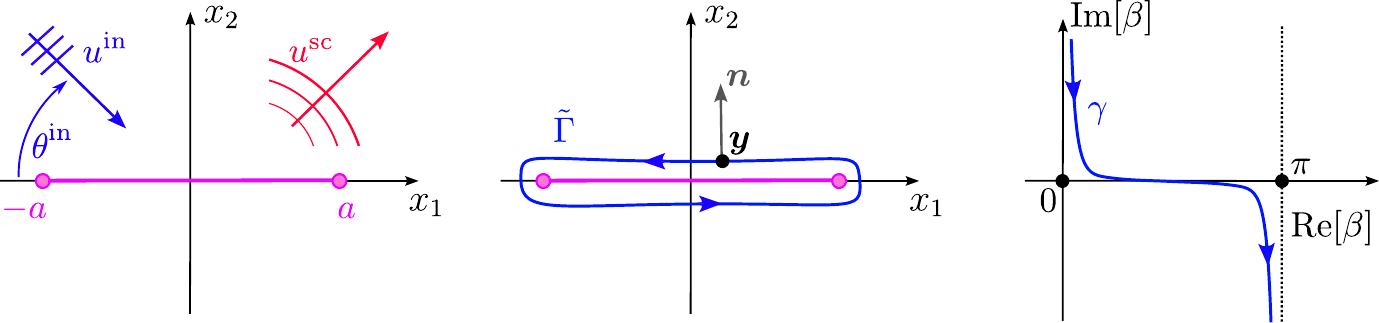}
	\caption{\small Geometry of the motivating planar problem (left). Illustration of the contour used in Green's identity (\ref{e:AS002})  (centre). Integration contour the plane wave representation (3.6) (right).}
	\label{f:AS01}
\end{figure}


\paragraph*{(d) Ordinary differential equation for the directivity. }

The points (a)--(c) are true for any generic compact scatterer with piecewise smooth boundary. A surprising fact can be established for the problem of diffraction by an ideal (Dirichlet or Neumann) segment considered in this section: 
the directivity $\mathcal{D}(\beta)$ can be shown (see e.g. \cite{Latta1956, Williams1982,Shanin2001}) to be a solution to a second-order ordinary differential equation of the form
\[
\left( \frac{d^2}{d\beta^2}  + F(\beta) \frac{d}{d\beta}  + 
H(\beta)\right) \mathcal{D}(\beta) = 0,  
\]
where the coefficients $F$ and $H$ are rather simple functions of $\beta$ (the equation
becomes Fuchsian after the change of variable $\xi = k \cos \beta$).

We plan to develop all the concepts listed above for spherical problems.
The notion of  plane-waves for scattering on a sphere was introduced in \cite{partAcomplexSphere}. The present 
article will focus on spherical analogies of the points (a)--(c), while the point (d) will be addressed in some future work. 





\section[Plane-wave representations of the ``free-space''  Green's function $G$]{Plane-wave representations of the ``free-space'' \newline Green's function $G$} \label{sec:FreeSpaceGreen}

In order to obtain plane-wave representations of the scattered field $u$  we will need
plane-wave representations of the Green's function $G(\x;\y)$ first. Plane-wave representations of the Green's function were a central point of \citep{partAcomplexSphere}. In this section, we will briefly summarise the main results of this paper and will also provide an additional, novel type of plane-wave representation for $G$. 

\subsection{The concept of plane wave on a sphere} \label{sec:planewavesdef}

First, let us introduce the ``sphere at
infinity'' $\Xi$, given in \citep[Section 2]{partAcomplexSphere}. To do this, we define the set $\mathcal{N}$ as the set of ``vectors of zero length'' in $\mathbb{C}^3$:
\begin{equation}
    \mathcal{N}=\{\bdeta=(\eta_1,\eta_2,\eta_3)\in\mathbb{C}^3 : \eta_1^2+\eta_2^2+\eta_3^2=0\}.
\end{equation}
For each such $\bdeta\in\mathcal{N}$, we define
\begin{equation}
    [\bdeta]\equiv\class{0:\eta_1 : \eta_2 : \eta_3} \in \CP^3,
\end{equation}
where $\mathbb{C}\mathbb{P}^3$ is the complex projective space of dimension 3. Here, $\class{0:\eta_1 : \eta_2 : \eta_3}$ denotes the equivalence class of the vector $(0,\eta_1, \eta_2, \eta_3)\in \mathbb{C}^4$ relative the following equivalence relation. $(0,\eta_1, \eta_2, \eta_3) \sim  (0,\eta'_1, \eta'_2, \eta'_3)$ if, and only if, $(0,\eta_1, \eta_2, \eta_3) = c \phantom{.} (0,\eta'_1, \eta'_2, \eta'_3)$  for some $c \in \mathbb{C}\setminus \{0\}$. 
Finally, the sphere at infinity $\Xi\subset\CP^3$ can be defined as 
\begin{equation}
  \Xi = \{ [\bdeta]  : \bdeta\in\mathcal{N}\}. \label{e:PWNew1.1}
\end{equation}

One may think of $\Xi$ in two different ways. First, $\Xi$ can be identified with the complex circle $\{(\cos(\beta),\sin(\beta)) \in \mathbb{C}^2: ~ \beta \in \mathbb{C}\}$ which has been compactified by ``adding in'' two infinite points located at $\beta = + i \infty$ and $\beta= - i \infty$. More precisely, as was shown in \citep{partAcomplexSphere}, \textit{every} point $[\bdeta]\in \Xi$ may be parametrised according to $
        [\bdeta] = [(\cos(\beta),\sin(\beta),i)],$ for complex $\beta$ so long as we permit $\beta = \pm i \infty$.
Second, 
$\Xi$ can be shown to be biholomorphic to the Riemann sphere $\widehat{\mathbb{C}}$ by using the change of variable $\tau=e^{-i\beta}$. So, one may think of $\Xi$ either as a closed cylinder, or as the Riemann sphere. This is illustrated in Figure \ref{fig:PWNew01}.

\begin{figure}[h!]
    \centering
    \includegraphics[width=0.95\linewidth]{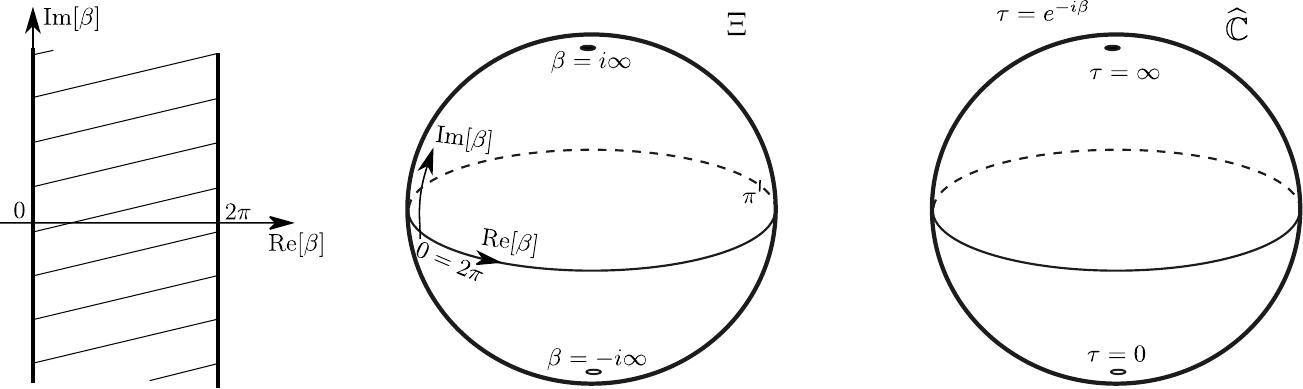}
    \caption{Visualisation of $\Xi$ as a closed complex circle (left and middle, taken from \citep{partAcomplexSphere}, Figure 2), and as Riemann sphere (right) in the variable $\tau = e^{- i \beta}$.}
    \label{fig:PWNew01}
\end{figure}
   Moreover, let us note that instead of choosing the parametrisation of $\Xi$ according to $[\bdeta] = [(\cos(\beta),\sin(\beta),i)]$, we may \textit{also} parametrise points on $\Xi$ via 
        \begin{align}
            [\bdeta] = [A_{\x}(\cos(\beta),\sin(\beta),i)] \label{e:ParamX}
        \end{align}
    where $A_{\x} \in \SO(3)$ denotes a real Euler rotation that rotates $\xNP=(0,0,1)$ onto some point $\x \in \mathcal{S}^2$. The parametrisation \eqref{e:ParamX} will play an important role in the Section \ref{sec:farfield}, and is also relevant to Sections \ref{sec:planewaverepG} and \ref{sec:PWNew02}, see Remark \ref{rem:4.4}.
    
    We will denote points in $\Xi$ by the letter $p$. A point $p\in \Xi$ can thus be thought of as an element of the quadric given in \eqref{e:PWNew1.1}, as an element of the complex circle, or as a point on the Riemann sphere. We will also write $p = [\boldsymbol{\eta}]$ to highlight that the definition \eqref{e:PWNew1.1} of $\Xi$ is used. 
    
    With the help of $\Xi$, we can define two different types of \textit{plane waves}.

\begin{definition}[\citep{partAcomplexSphere}, Definition 3.1]
  Let $p \in \Xi$ be a point represented by $p = [\boldsymbol{\eta}]$. A plane
  wave for Laplace-Beltrami scattering problems on
  $\mathcal{S}^2$ is given by either of the following two functions:
  \begin{equation}
    w_1 (\boldsymbol{x}, p ; \boldsymbol{x}^{\mathrm{ref}}) = \left( - i
    \frac{\boldsymbol{x} \cdot \boldsymbol{\eta}}{\boldsymbol{x}^{\mathrm{ref}} \cdot
    \boldsymbol{\eta}} \right)^{- (1 + \lambda)}, \quad w_2 (\boldsymbol{x}, p ;
    \boldsymbol{x}^{\mathrm{ref}}) = \left( - i \frac{\boldsymbol{x} \cdot
    \boldsymbol{\eta}}{\boldsymbol{x}^{\mathrm{ref}} \cdot \boldsymbol{\eta}}
    \right)^{\lambda},
    \label{def:planewavesderaph}
  \end{equation}
  where $\boldsymbol{x}, \boldsymbol{x }^{\mathrm{ref}} \in \mathcal{S}^2$ and
  $\boldsymbol{x}^{\mathrm{ref}}$ is a fixed reference point. We choose the branch of these plane waves according to
  \begin{equation}
    w_1 (\boldsymbol{x}^{\mathrm{ref}}, p ; \boldsymbol{x}^{\mathrm{ref}}) = e^{i \pi
    (\lambda + 1) / 2}, \quad w_2 (\boldsymbol{x}^{\mathrm{ref}}, p ;
    \boldsymbol{x}^{\mathrm{ref}}) = e^{- i \pi \lambda / 2} .
    \label{eq:planewavesbranch}
  \end{equation}
  \label{def:planewavesbranch}
\end{definition}
As explained in \cite{partAcomplexSphere}, these functions are called plane waves because they are shown to be giving the behaviour of the Green's function $G(\x;\y)$ as the source $\y$ tends to infinity within the complexification of $\S^2$. The point $p\in\Xi$ does in some way specify in which `direction' this limit is taken.

\subsection{Singularities of plane-waves; the maps $\Phi$ and $\Psi$} \label{sec:singplanewaves}

As shown in \cite[Section 3]{partAcomplexSphere}, the plane waves $w_1$ and $w_2$ exhibit several singularities. For fixed $\boldsymbol{x}
= \boldsymbol{x}^{\star} \in \mathcal{S}^2$ they are singular on $\Xi$ whenever $p=[\bdeta]$ is such that
$\boldsymbol{\boldsymbol{x}^{\star}} \cdot \boldsymbol{\eta} = 0$ or
$\boldsymbol{x}^{\mathrm{ref}} \cdot \boldsymbol{\eta} = 0$. For fixed
$p=p^{\star} = [\boldsymbol{\eta}^{\star}] \in \Xi$ chosen such that $\x^{\mathrm{ref}}\cdot \boldsymbol{\eta}^{\star} \neq 0$, they are singular on
$\mathcal{S}^2$ whenever $\boldsymbol{x} \cdot \boldsymbol{\eta }^{\star} = 0$. These singularities are all of the branching type. Because of the explicit definition of these plane waves in terms of elementary functions, it is direct to see that their gradient (w.r.t.\ $\x$) also has the same singularity structure. In particular the same is true for any normal derivatives that will be taken later on in Section~\ref{sec:PWNew02}.


For $\boldsymbol{x} \in \mathcal{S}^2$, the values of $\bdeta$ that are solutions to the equation $\boldsymbol{x} \cdot
\boldsymbol{\eta} = 0$ correspond to two points on $\Xi$ which we denote by
$\Phi_+ (\boldsymbol{x})$ and $\Phi_- (\boldsymbol{x})$. The points $\Phi_+
(\boldsymbol{x})$ and $\Phi_- (\boldsymbol{x})$ vary smoothly with $\boldsymbol{x}$,
so we only need to specify the points $\Phi_+ (\boldsymbol{x})$ and $\Phi_-
(\boldsymbol{x})$ at a single point $\boldsymbol{x} \in \mathcal{S}^2$ in order to
define which points correspond to the $_+$ subscript, and which points
correspond to the $_-$ subscript, see \citep[Lemma 2.5, $f$.]{partAcomplexSphere}.
This is done as follows. We set
\begin{equation}
  \Phi_+ (0, 0, 1) = [(- i, 1, 0)], \quad \text{and} \quad \Phi_- (0, 0, 1) =
  [(- i, - 1, 0)] .
\end{equation}
Thus, we should think of $\Phi (\boldsymbol{x}) = \{ \Phi_+ (\boldsymbol{x}), \Phi_-
(\boldsymbol{x}) \}$ as a (2-valued) mapping from $\mathcal{S}^2 \rightarrow
\Xi$, as illustrated in Figure \ref{fig:PWApp03}.

In Figure \ref{fig:PWApp03} we also show the 2-valued map $\Psi : \Xi
\rightarrow \mathcal{S}^2$, which assigns to $p=[\bdeta]$ the two points $\Psi_+ (p)$
and $\Psi_- (p)$ corresponding to the two solutions of the equation $\boldsymbol{x}
\cdot \boldsymbol{\eta} = 0$ on $\mathcal{S}^2$.

\begin{figure}[h!]
    \centering
    \includegraphics[width=0.6\linewidth]{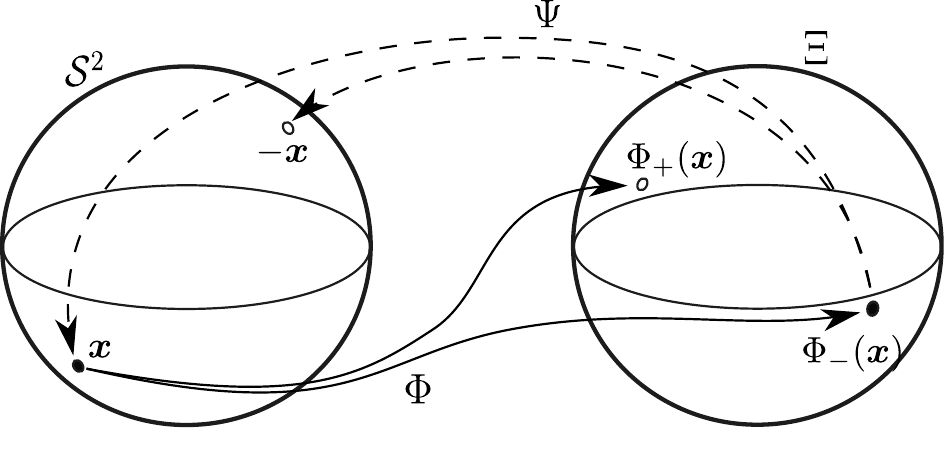}
    \caption{\small Illustration of the maps $\Psi$ and $\Phi$, adapted from \citep{partAcomplexSphere}.}
    \label{fig:PWApp03}
\end{figure}

For any given $p\in \Xi$, the points $\Psi_+ (p)$ and $\Psi_- (p)$ are opposites, in
the sense that $\Psi_+ (p) = - \Psi_- (p)$, see Figure \ref{fig:PWApp03}. The points $\Phi_+ (\boldsymbol{x})$ and $\Phi_- (\boldsymbol{x})$ are also
``opposite'' on $\Xi$. Formulating precisely what it means for \
$\Phi_+ (\boldsymbol{x})$ and $\Phi_- (\boldsymbol{x})$ to be opposite is quite
lenghty, however. We refer to \citep[Section2]{partAcomplexSphere}, for a detailed
explanation. Here, we just note the following properties of $\Phi_{\pm}$ and $\Psi_{\pm}$. 
\begin{enumerate}
    \item For all $\x \in \S^2$, we have $
  \Psi_+ (\Phi_+ (\boldsymbol{x})) = \boldsymbol{x}$ and $\Psi_-
  (\Phi_- (\boldsymbol{x})) = \boldsymbol{x}$;
  \item For all $p\in\Xi$, we have $\Phi_+(\Psi_+(p))=p$ and $\Phi_-(\Psi_-(p))=p$. In particular, this implies ${\Phi_+(-\x)= \Phi_-(\x)}$;
  \item The maps $\Psi_+,\Psi_- : \Xi \to \mathcal{S}^2$ are diffeomorphisms, with $\Phi_+$ and $\Phi_-$, respectively, being their inverse mappings; 
  \item For any great circle $\Gamma \subset\mathcal{S}^2$, we have $\Phi_+(\Gamma) = \Phi_-(\Gamma)$, which we will just denote~$\Phi(\Gamma)$. If one thinks of $\Xi$ as the Riemann sphere $\widehat{\mathbb{C}}$, one should think of $\Phi(\Gamma)$ as a great circle on this Riemann sphere, as illustrated in Figure \ref{fig:PWApp01}, right. 
\end{enumerate}
Explicit formulae for all these maps are given in \citep{partAcomplexSphere}. Using the properties 2, 3, and 4 above, we find the following Lemma to be true.

\begin{lemma}
    Let $\x$ and $\y$ be two distinct points of $\S^2$ and let $\Gamma$ be a great circle of $\S^2$ separating $\x$ and $\y$. Then the contour $\gamma=\Phi(\Gamma)$ is a contour on $\Xi$ separating the two points $\Phi_+(\y)$ and $\Phi_-(\x)$ from the two points $\Phi_-(\y)$ and $\Phi_+(\x)$.
    \label{lem:defgammaonXi}
\end{lemma}


Let us note that the link established by the maps $\Phi$ and $\Psi$ between $\mathcal{S}^2$ and $\Xi$ is due to a  geometric structure common to both spheres, as was discussed in \citep[Section 2]{partAcomplexSphere}. This structure is the \textit{complexified sphere}  $\mathcal{S}^2_{\c}$, of which $\mathcal{S}^2$ is a submanifold (see also Section \ref{sec:analytical-continuation}). $\Xi$, on the other hand, may be interpreted as an ``extremity'' of $\mathcal{S}^2_{\c}$ in the sense, that it can be viewed as a sphere describing the (asymptotic) infinity of $\mathcal{S}^2_{\c}$.

\subsection{Plane-wave representations of the Green's function} \label{sec:planewaverepG}

Since $\Xi$ is biholomorphic to the Riemann sphere
$\widehat{\mathbb{C}}$, one can use contour integration on
$\Xi$ and all of the usual
one-complex-variable concepts become available on $\Xi$. For instance, branching singularities on $\Xi$ are just usual 1D complex analysis branch points, and one might artificially introduce some branch cuts linking those to ensure single-valuedness of singular functions (e.g.\ the plane waves) on $\Xi$ minus the cuts. Moreover, as explained in \cite[Appendix D]{partAcomplexSphere}, a meromorphic 1-form $\Omega$ on $\Xi$ with only simple poles is uniquely determined by its residues. We can therefore define the form $\Omega_{\x}$ as follows.

\begin{definition}
  Let $\boldsymbol{x} \in \mathcal{S}^2$ be fixed. We define the 1-form $\Omega_{\boldsymbol{x}} $ on $\Xi$ as the
  unique meromorphic 1-form with one simple pole at $p=\Phi_+
  (\boldsymbol{x})$ with residue $-i$, another simple pole at $p=\Phi_- (\boldsymbol{x})$ with residue $+i$, and no other poles. 
\end{definition}

\begin{remark}\label{rem:4.4}
Let us give an explicit formula for $\Omega_{\x}$. Parametrise points $p \in \Xi$ according to \eqref{e:ParamX}. Then, in this coordinate system, we have 
     $\Omega_{\x} = d\beta.$
Indeed, $d\beta$ is singular at $\beta = \pm i \infty$, and the points $p$ corresponding to $\beta = \pm i \infty$ are actually equal to $\Phi_{\pm}(\x)$, as follows from \citep[Sections 2(b)-(c)]{partAcomplexSphere}. The fact that $ \Omega_{\x} = d\beta$ has residue $-i$ at $\beta = i \infty$ and residue $+i$ at $\beta = -i \infty$ can be seen by choosing a new variable $\tau_{\pm} = e^{\pm i \beta}$ on $\Xi$, see \citep[Section 4(a)]{partAcomplexSphere}. 
\end{remark} 


Contours akin to $\gamma$ in Lemma \ref{lem:defgammaonXi} and the 1-form
$\Omega_{\boldsymbol{x}}$ were introduced and discussed in detail in \citep[Section 4 and Appendix E]{partAcomplexSphere}.  They are key to the following proposition that provides the first two plane-wave representations for the ``free-space'' Green's function $G$ introduced in (\ref{eq.Ch02--GreensFunctionGPDE})-(\ref{eq.Ch02--Def.GreensFunctionG}). This result was initially derived and proven in \citep{partAcomplexSphere}.


\begin{proposition}[\cite{partAcomplexSphere}, Proposition D.2]
    Let $\x$ and $\y$ be two points on~$\S^2$ such that $\x\neq\y$, and define $\mathcal{B}_{\lambda}=\tfrac{-i}{8 \pi \sin(\pi \lambda)} = \mathcal{B}_{-\lambda-1}$. Then, the following integral representation formulae are valid
    \begin{align}
         G (\boldsymbol{x} ; \boldsymbol{y}) &=\mathcal{B}_{\lambda}
    \int_{\gamma} w_2  (\boldsymbol{x}, p ; -
    \boldsymbol{x}^{\mathrm{ref}}) w_1 (\boldsymbol{y}, p ;
    \boldsymbol{x}^{\mathrm{ref}}) \Omega_{\boldsymbol{x}^{\mathrm{ref}}},
    \label{e:New01} \\
     G (\boldsymbol{x} ; \boldsymbol{y}) &= \mathcal{B}_{\lambda}
    \int_{\gamma} w_1 (\boldsymbol{x}, p ; -
    \boldsymbol{x}^{\mathrm{ref}}) w_2 (\boldsymbol{y}, p ;
    \boldsymbol{x}^{\mathrm{ref}}) \Omega_{\boldsymbol{x}^{\mathrm{ref}}}
    \label{e:New02},
    \end{align}
    where $\gamma$ is any contour of the form $\Phi(\Gamma)$ for some great circle $\Gamma$ on $\S^2$ separating the two points $\x$ and $\y$, and $\x^{\rm ref}$ is a reference point lying on the same side of $\Gamma$ as $\y$. This is illustrated in Figure \ref{fig:PWApp01}~(left and middle). 
    The contour $\gamma$ is understood to be oriented such that the points
  $\Phi_+ (\boldsymbol{y})$ and $\Phi_- (\boldsymbol{x})$ lie to its left, and the points $\Phi_-
  (\boldsymbol{y})$ and $\Phi_+ (\boldsymbol{x})$ lie to its  right. The integration is performed with respect to $p$.
  \label{lem:PWNew01}
\end{proposition}

Interestingly, there exists a third type of plane-wave representation of $G(\x;\y)$, which was not mentioned in \citep{partAcomplexSphere}, and is given below. It is proved in Appendix~\ref{App:New01}.

\begin{proposition} \label{prop:gammainfinityG}
     Let $\x$, $\y$ and $\x^{\text{ref}}$ be points on $\S^2$ such that $\x\neq\pm\y$ and $\x\neq \pm\x^{\text{ref}}$. Then, the following formula is valid
{\fontsize{11.85pt}{10}\selectfont
     \begin{align}
    G (\boldsymbol{x}; \boldsymbol{y}) &= \frac{\mathcal{B}_{\lambda}}{2}
   \!\! \int_{\gamma_{\infty}} \!\!\!\!\! \left\{w_1 (\boldsymbol{y}, p ;
    \boldsymbol{x}^{\mathrm{ref}}) w_2 (\boldsymbol{x}, p ; -
    \boldsymbol{x}^{\mathrm{ref}}) - w_2 (\boldsymbol{y}, p ;
    \boldsymbol{x}^{\mathrm{ref}}) w_1 (\boldsymbol{x}, p ; -\boldsymbol{x}^{\mathrm{ref}})\!\right\}\!\Omega_{\boldsymbol{x}^{\mathrm{ref}}},
    \label{e:NewPW23}
  \end{align}} 
  where $\gamma_{\infty}$ is any figure-of-eight loop as illustrated in Figure \ref{fig:PWApp01}, right. It is oriented such that it bypasses  $\Phi_-(\x)$ in the positive direction, and  $\Phi_+(\x)$ in the negative direction. Since the integrand in \eqref{e:NewPW23} is branching as $\gamma_{\infty}$ winds around $\Phi_{\pm}(\x)$, we need to specify the branch of plane waves in \eqref{e:NewPW23}. From the discussion of Appendix \ref{App:New01}, we find that at the point $p_0$, shown in Figure \ref{fig:PWApp01}, right, the principal branch introduced in Definition \ref{def:planewavesbranch} is taken. Note that specifying the location of $p_0$ as shown in Figure \ref{fig:PWApp01} determines that the (oriented) segment of $\gamma_{\infty}$ that starts close to $\Phi_+(\x)$ and ends close to $\Phi_-(\x)$ lies on this principal branch.  
\end{proposition}

\begin{figure}[h!]
    \centering
    \includegraphics[width=\linewidth]{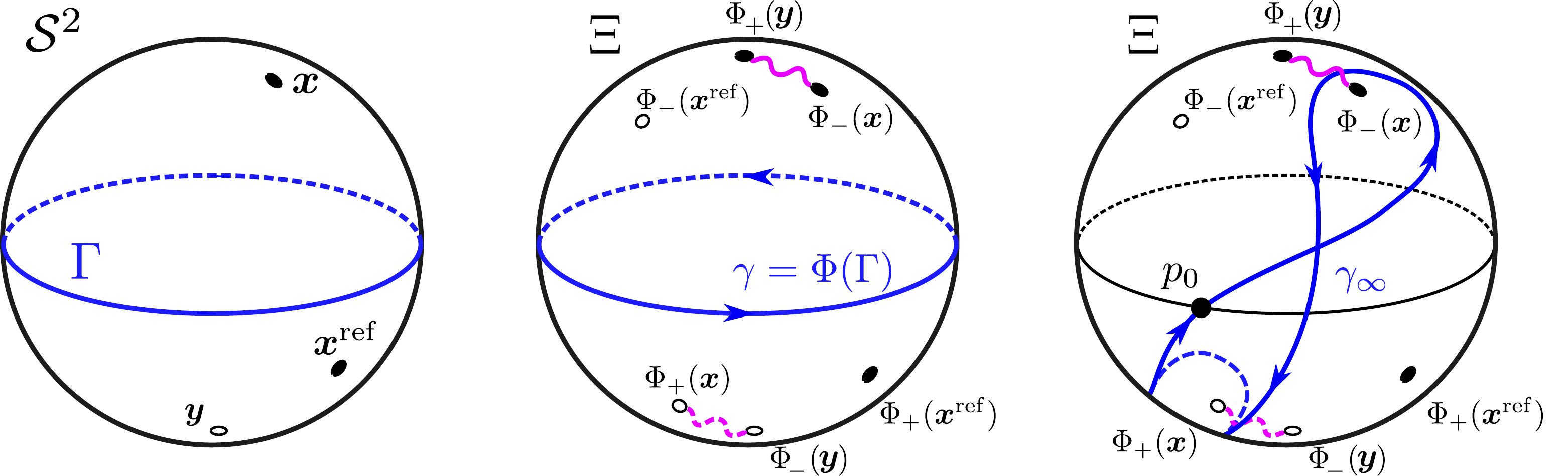}
    \caption{\small Illustration of a great circle $\Gamma \subset \mathcal{S}^2$ (left),  the associated contour $\gamma \subset \Xi$ (middle) and a typical contour $\gamma_\infty\subset\Xi$ (right). We also see branch cuts of the plane-waves, connecting $\Phi_+(\y)$ to $\Phi_-(\x)$ and $\Phi_+(\x)$ to $\Phi_-(\y)$, as explained in \citep{partAcomplexSphere}.}
    \label{fig:PWApp01}
\end{figure}

It is interesting to note that the conditions of validity of \eqref{e:NewPW23} ($\x\neq\pm\y$)  seem more restrictive than those of \eqref{e:New01} and \eqref{e:New02} ($\x\neq\y$), but that there are less restrictions on the choice of $\x^{\rm ref}$ for \eqref{e:NewPW23} than there are for \eqref{e:New01} and \eqref{e:New02}. Moreover, we see at the end of Section \ref{sec:PWNew02} that the third plane-wave representation has some other clear advantages. 

Finally, note that the contours $\gamma$ and $\gamma_\infty$ introduced in this section do implicitly depend on $\x$, $\y$ and $\x^{\text{ref}}$. 

We will now show that similar plane-wave representations can be derived for the scattered field $u$.

\section{Plane wave representation of the scattered field $u$} \label{sec:PWNew02}

In this Section, we will use the various plane-wave representations of $G(\x;\y)$ given within \eqref{e:New01}, \eqref{e:New02}, and \eqref{e:NewPW23} together with Green's third identity in order to obtain three different plane-wave representations of the sought-after scattered wave-field $u(\x)$ introduced in Section \ref{sec:pbformulationsphere} and solution to the boundary value problem \eqref{e:HLB}-\eqref{e:ID}. Throughout this section, we assume that the spherical cap $\mathcal{C}$ and the scatterer $\mathbb{S}$ fully contained within $\mathcal{C}$ are as defined in Section \ref{sec:pbformulationsphere}. We start by the following Lemma, which is an important yet direct consequence from Proposition \ref{lem:PWNew01}.



\begin{lemma}\label{lem:PWNew1.5}
    Assume that $\x\in\S^2\setminus\mathcal{C}$. Then there exists a great circle $\Gamma$ of $\S^2$ that separates $\x$ from $\mathcal{C}$. Moreover, 
    \begin{align}
        \forall \y\in\partial\mathbb{S}, \ \forall \x^{\mathrm{ref}}\in\mathbb{S}, \   G (\boldsymbol{x} ; \boldsymbol{y}) &= \mathcal{B}_{\lambda}
  \int_{\gamma} w_2  (\boldsymbol{x}, p ; - \boldsymbol{x}^{\mathrm{ref}}) w_1
  (\boldsymbol{y}, p ; \boldsymbol{x}^{\mathrm{ref}})
  \Omega_{\boldsymbol{x}^{\mathrm{ref}}}, \label{eq:first_GreenLem}
   \\
  \forall \y\in\partial\mathbb{S}, \ \forall \x^{\mathrm{ref}}\in\mathbb{S}, \ G (\boldsymbol{x} ; \boldsymbol{y}) &= \mathcal{B}_{\lambda}
     \int_{\gamma} w_1 (\boldsymbol{x}, p ; -
     \boldsymbol{x}^{\mathrm{ref}}) w_2 (\boldsymbol{y}, p ;
     \boldsymbol{x}^{\mathrm{ref}}) \Omega_{\boldsymbol{x}^{\mathrm{ref}}}, \label{eq:second_GreenLem}
    \end{align}
    where $\gamma=\Phi(\Gamma)$ is oriented such that  $\Phi_-(\x)$ and $\Phi_+\left (\partial \mathbb{S}\right )$ lie to its left and is completely independent from $\y$ or $\x^{\mathrm{ref}}$. This is illustrated in Figure \ref{fig:PWRep01}, right. 
\end{lemma}

\begin{proof} The existence of such a great circle $\Gamma$ is a purely geometric result. Because $\mathcal{C}$ is fully contained within the Northern hemisphere,  if $\x$ lies on the Southern hemisphere, we can take $\Gamma$ to be the equator. Let us now give a suitable great circle when $\x$ is in the Northern hemisphere. Denote by $(\theta_{\x},\varphi_{\x})$  the spherical coordinates of $\x$, and let $\theta_{\mathcal{C}}$ be the polar angle associated with $\partial \mathcal{C}$. Let $\hat{\x}$ denote the point with spherical coordinates ${(\tfrac{1}{2}(\theta_{\mathcal{C}}+\theta_{\x}),\varphi_{\x})}$. Within the $x_3 \equiv 0$ plane, consider the line $L$ obtained by intersecting the plane orthogonal to $\hat{\x}$ with the $x_3 \equiv 0$ plane.
Then, we  can take $\Gamma$ to be the unique great circle passing through the point $\hat{\x}$ and through any of two points where $L$ intersects $\mathcal{S}^2$, as illustrated in Figure \ref{fig:PWRep01}, right.
This existence result is crucial and one of the main reason for introducing $\mathcal{C}$ in the first place.

 To conclude the proof, note that for all $\y\in\partial \mathbb{S}$ and all $\x^{\mathrm{ref}}\in\mathbb{S}$, the great circle  $\Gamma$ does separate $\x$ from both $\y$ and $\x^{\mathrm{ref}}$ and therefore the result of Proposition \ref{lem:PWNew01} can be applied.
\end{proof}
Now,  according to Green's third identity \eqref{e:Greens3rdSphere}, by choosing $\tilde \Gamma=\partial\mathbb{S}$, for any  $\x\in\S^2\setminus\mathcal{C}$, we have 
 \begin{equation}
   u (\boldsymbol{x}) = \int_{\partial \mathbb{S}} \left[ u (\boldsymbol{y}) \partial_{\boldsymbol{n}} G
   (\boldsymbol{x} ; \boldsymbol{y}) - G (\boldsymbol{x} ; \boldsymbol{y})
   \partial_{\boldsymbol{n}} u (\boldsymbol{y})\right] d\ell.\label{e:New002}
 \end{equation}

This will allow us to derive the following plane-wave representations for $u$. 

\begin{proposition}\label{prop:5.2}
    For any $\x\in\S^2\setminus\mathcal{C}$ and any $\x^{\mathrm{ref}}\in\mathbb{S}$, we have
    \begin{align}
  u (\boldsymbol{x}) &= \mathcal{B}_{\lambda} \int_{\gamma}
  \mathcal{D}_1 (p ; \boldsymbol{x }^{\mathrm{ref}}) w_2 (\boldsymbol{x}, p ; -
  \boldsymbol{x }^{\mathrm{ref}}) \Omega_{\boldsymbol{x }^{\mathrm{ref}}}, 
  \label{e:New13} \\
  u (\boldsymbol{x}) &= \mathcal{B}_{\lambda} \int_{\gamma}
   \mathcal{D}_2 (p ; \boldsymbol{x }^{\mathrm{ref}}) w_1 (\boldsymbol{x}, p ; -
  \boldsymbol{x }^{\mathrm{ref}}) \Omega_{\boldsymbol{x }^{\mathrm{ref}}}, 
  \label{e:New014}
\end{align}
where $\gamma$ is defined as in Lemma \ref{lem:PWNew1.5} and the \emph{directivities} $ \mathcal{D}_1$ and $ \mathcal{D}_2$ are given by
\begin{align}
   \mathcal{D}_1 (p ; \boldsymbol{x }^{\mathrm{ref}}) &= \int_{\partial \mathbb{S}} \left[u (\boldsymbol{y})
  \partial_{\boldsymbol{n}} w_1 (\boldsymbol{y}, p ; \boldsymbol{x }^{\mathrm{ref}}) -
  w_1 (\boldsymbol{y}, p ; \boldsymbol{x }^{\mathrm{ref}}) \partial_{\boldsymbol{n}} u
  (\boldsymbol{y})\right] d\ell,  \label{e:New14} \\
    \mathcal{D}_2 (p ; \boldsymbol{x }^{\mathrm{ref}}) &= \int_{\partial \mathbb{S}} \left[u (\boldsymbol{y})
  \partial_{\boldsymbol{n}} w_2 (\boldsymbol{y}, p ; \boldsymbol{x }^{\mathrm{ref}}) -
  w_2 (\boldsymbol{y}, p ; \boldsymbol{x }^{\mathrm{ref}}) \partial_{\boldsymbol{n}} u
  (\boldsymbol{y})\right] d\ell. \label{e:New18}
\end{align}
\end{proposition}

\begin{proof}
    Using \eqref{eq:first_GreenLem} within \eqref{e:New002}, we obtain
    \begin{equation*}
  u (\boldsymbol{x}) = \mathcal{B}_{\lambda} \int_{\partial \mathbb{S}}
  \left[ \int_{\gamma}  \left \{(u (\boldsymbol{y}) \partial_{\boldsymbol{n}} w_1
  (\boldsymbol{y}, p ; \boldsymbol{x }^{\mathrm{ref}}) - w_1 (\boldsymbol{y}, p ;
  \boldsymbol{x }^{\mathrm{ref}}) \partial_{\boldsymbol{n}} u (\boldsymbol{y})\right \} w_2
  (\boldsymbol{x}, p ; \boldsymbol{-} \boldsymbol{x}^{\mathrm{ref}}) \Omega_{\boldsymbol{x
  }^{\mathrm{ref}}}\right] d\ell \label{e:New12}.
\end{equation*}
Now, since $\gamma$ is the exact same contour for all possible $\y\in\partial \mathbb{S}$, we can swap the order of integration to get
\begin{equation*}
  u (\boldsymbol{x}) = \mathcal{B}_{\lambda} \int_{\gamma} 
  \left\{ \int_{\partial \mathbb{S}}  \left [(u (\boldsymbol{y}) \partial_{\boldsymbol{n}} w_1
  (\boldsymbol{y}, p ; \boldsymbol{x }^{\mathrm{ref}}) - w_1 (\boldsymbol{y}, p ;
  \boldsymbol{x }^{\mathrm{ref}}) \partial_{\boldsymbol{n}} u (\boldsymbol{y})\right ] w_2
  (\boldsymbol{x}, p ; \boldsymbol{-} \boldsymbol{x}^{\mathrm{ref}}) d\ell \right\}  \Omega_{\boldsymbol{x
  }^{\mathrm{ref}}}\label{e:New12}.
\end{equation*}
Since $ w_2
  (\boldsymbol{x}, p ; \boldsymbol{-} \boldsymbol{x}^{\mathrm{ref}})$ does not depend on $\y$ we can take it out of the inner integral to obtain \eqref{e:New13}. The plane-wave representation \eqref{e:New18} is obtained in the exact same way but starting from \eqref{eq:second_GreenLem} instead of \eqref{eq:first_GreenLem}.
\end{proof}


\begin{figure}[h!]
    \centering
    \includegraphics[width=\linewidth]{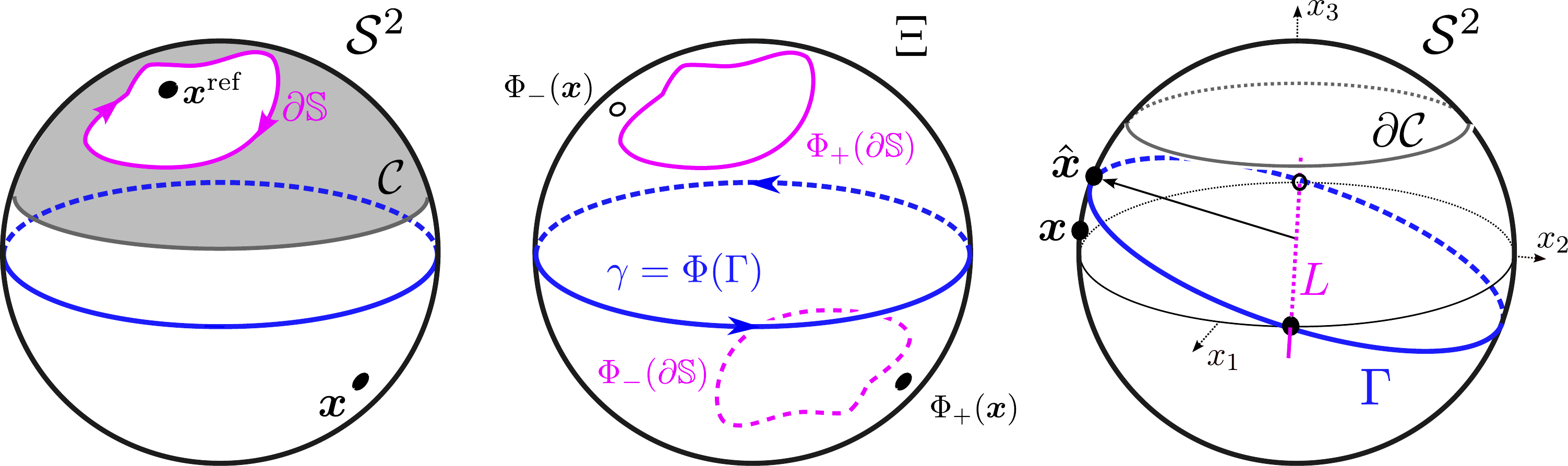}
    \caption{\small Left: Contours $\Gamma$ and $\partial \mathbb{S}$. We also see that $\Gamma$ separates $\x$ and $\mathcal{C}$ (and therefore $\partial \mathbb{S}$), and that $\x^{\mathrm{ref}}$ lies within $\mathbb{S}$, as required for the validity of formula \eqref{e:New12}. Middle: Contours $\gamma = \Phi(\Gamma)$, $\Phi_{\pm}(\partial{\mathbb{S}})$ as well as the points $\Phi_{\pm}(\x)$ on $\Xi$. Right: Tilted contour $\Gamma$ used in the proof of Lemma \ref{lem:PWNew1.5} when $\x$ is on the northern hemisphere.} 
    \label{fig:PWRep01}
\end{figure}

The directivities introduced within \eqref{e:New14} and \eqref{e:New18} are analogous to those encountered in the planar setting in the motivational Section \ref{sec:motivation}, as we will explain in Section \ref{sec:analogies}.
 
\begin{remark}\label{rem:PWNew2.3}
     It can be checked that the integrands in \eqref{e:New14} and \eqref{e:New18} are singular on $\mathcal{S}^2$ \textit{only if} $\Psi_+(p)$ or $\Psi_-(p)$ hits the contour $\partial \mathbb{S}$ of integration. This is because, for ${\x^{\mathrm{ref}} \in \mathcal{S}^2}$ fixed, $\Psi_{\pm}(p)$ describe the locations of the singularities of $w_2(\y,p; \x^{\mathrm{ref}})$ and its normal derivative on $\mathcal{S}^2$. However, as we can see from Figure \ref{fig:PWRep01}, the contours $\partial \mathbb{S}$ and $\gamma=\Phi(\Gamma)$ are chosen such that $\Gamma \cap \partial \mathbb{S} = \emptyset$. So, the directivities given in \eqref{e:New14} and \eqref{e:New18} are indeed well-defined for $p \in \gamma$. Note further that the directivities can be seen as integrals that depend on the parameter $p$. Since the integrand is analytic in $p$ (at least in the neighbourhood of $\gamma$), then these quantities can be analytically continued off $\gamma$ as long as $p$ does not hit any of the singularities of $w_{1,2}$ and their normal derivatives.
\end{remark}

Our next goal is to extend the region of validity of the plane-wave representations given in \eqref{e:New13} and \eqref{e:New014} from $\mathcal{S}^2 \setminus \mathcal{C}$ onto all of $\mathcal{S}^2 \setminus \mathbb{S}$. This is achieved with the following Lemma. 

\begin{lemma} \label{cor:New02} Consider again that $\x^{\mathrm{ref}}$ is in $\mathbb{S}$ but that $\x$ is now in $\S^2\setminus\mathbb{S}$. There exists a contour
  $\tilde \gamma$ on $\Xi$ such that $u(\x)$ can be written as either of the following two
  plane-wave representations
  \begin{equation}
    u (\boldsymbol{x}) = \mathcal{B}_{\lambda} \int_{\tilde \gamma} \mathcal{D}_2 (p ; \boldsymbol{x }^{\mathrm{ref}}) w_1
    (\boldsymbol{x}, p ; - \boldsymbol{x }^{\mathrm{ref}}) \Omega_{\boldsymbol{x
    }^{\mathrm{ref}}}, \quad \x \in \mathcal{S}^2 \setminus \mathbb{S} \label{e:New16}
  \end{equation}
  \begin{equation}
    u (\boldsymbol{x}) = \mathcal{B}_{\lambda} \int_{\tilde \gamma} \mathcal{D}_1 (p ; \boldsymbol{x }^{\mathrm{ref}}) w_2
    (\boldsymbol{x}, p ; - \boldsymbol{x }^{\mathrm{ref}}) \Omega_{\boldsymbol{x
    }^{\mathrm{ref}}}, \quad  \x \in \mathcal{S}^2 \setminus \mathbb{S},  \label{e:New17}
  \end{equation}
  where the directivities  $\mathcal{D}_1$ and $\mathcal{D}_2$ are still given by \eqref{e:New14} and \eqref{e:New18}.
\end{lemma}

\begin{proof}
  According to \eqref{e:New13} and \eqref{e:New014}, whenever $\x\in \mathcal{S}^2 \setminus \mathcal{C}$, we can take  $\tilde \gamma = \gamma$. 
  So, it only remains to show that the plane-wave representations given in
  \eqref{e:New13} and \eqref{e:New014} can be extended to account for general
  observation directions $\boldsymbol{x} \in \mathcal{S}^2 \setminus
  \mathbb{S}$.  

\begin{figure}[h!]
    \centering
    \includegraphics[width=0.75\linewidth]{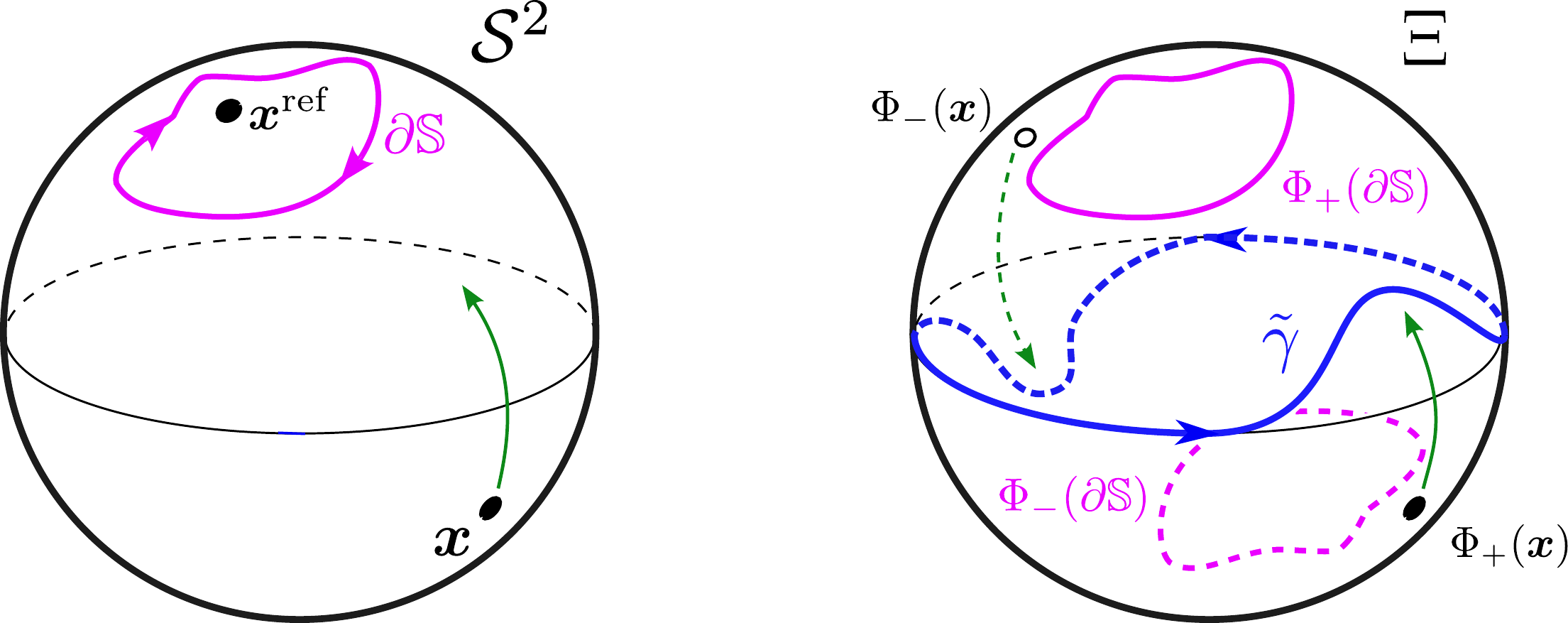}
    \caption{\small Illustration of the procedure carried out within the proof of Lemma \ref{cor:New02}: On the left, we see how $\x$ varies on $\mathcal{S}^2$ indicated by an arrow. On the right, we see the new contour $\tilde{\gamma}$ and how the singularities located at $\Phi_{\pm}(\x)$ vary on $\Xi$.}
    \label{fig:PWRep10}
\end{figure}

  According to Remark \ref{rem:PWNew2.3} and what we know about the plane waves singularities (Section \ref{sec:singplanewaves}), the directivities $\mathcal{D}_{1,2}(p ; \boldsymbol{x}^{\mathrm{ref}})$, as functions of $p$, are analytic on ${\Xi \setminus (\Phi_+(\mathbb{S})\cup\Phi_-(\mathbb{S})}$. 
  So, we can now let $\boldsymbol{x}$ continuously vary on $\mathcal{S}^2
  \setminus \mathbb{S}$ and use Cauchy's theorem to continuously deform the
  contour $\gamma$ on $\Xi \setminus (\Phi_+(\mathbb{S})\cup\Phi_-(\mathbb{S})$. This deformation needs to be done
  such that $\Phi_+(\boldsymbol{x})$ and $\Phi_-(\boldsymbol{x})$ never hit the contour. The  contour resulting from this continuous deformation process is what we will refer
  to as $\tilde \gamma$. This is illustrated in Figure~\ref{fig:PWRep10}.
  \end{proof}

  \begin{remark}
   Importantly, note that as $\x$ varies on $\mathcal{S}^2$, $\Phi_+
  (\boldsymbol{x})$ is allowed to pass through $\Phi_- (\mathbb{S})$ and $\Phi_-
  (\boldsymbol{x})$ is allowed to pass through $\Phi_+ (\mathbb{S})$ as this
  hits no singularities of the integrands
  \[ \mathcal{D}_2 (p ; \boldsymbol{x }^{\mathrm{ref}}) w_1 (\boldsymbol{x}, p ; -
     \boldsymbol{x }^{\mathrm{ref}}) \quad \text{or} \quad \mathcal{D}_1 (p ;
     \boldsymbol{x }^{\mathrm{ref}}) w_2 (\boldsymbol{x}, p ; - \boldsymbol{x
     }^{\mathrm{ref}}) \]
  within \eqref{e:New16} or \eqref{e:New17}.
\end{remark}




Next, let us use the third plane-wave representation of $G(\x;\y)$ given in \eqref{e:NewPW23} in order to obtain a plane-wave representation of $u$ involving the contour $\gamma_\infty$. This is achieved in a fashion similar to how we proved Proposition~\ref{prop:5.2} and Lemma~\ref{cor:New02}, so we will be brief.  

Let us start by considering $\x\in\S^2\setminus \mathcal{C}$ and $\x^{\mathrm{ref}} \in \mathbb{S}$. In analogy with Lemma \ref{lem:PWNew1.5}, we can ensure that there exists a contour $\gamma_\infty$ independent of $\y$, such that \eqref{e:NewPW23} is valid for all $\y\in\partial \mathbb{S}$. 

Note that because of the added constraint $\y\neq\x$ in Proposition \ref{prop:gammainfinityG}, we should actually restrict $\x$ a bit more and can only consider $\x\in\S^2\setminus(\mathcal{C}\cup(-\mathbb{S}))$, where $-\mathbb{S}$ is the antipodal image of the scatterer. This representation can then be used within Green's third identity \eqref{e:New002}, and the order of integration can be swapped to obtain the representation
\begin{equation}
     u (\boldsymbol{x}) =\frac{\mathcal{B}_{\lambda}}{2} \int_{\gamma_\infty} \left\{\mathcal{D}_1 (p ; \boldsymbol{x}^{\mathrm{ref}}) w_2 (\boldsymbol{x},
  p ; - \boldsymbol{x}^{\mathrm{ref}}) - \mathcal{D}_2(p, \boldsymbol{x}^{\mathrm{ref}})
  w_1 (\boldsymbol{x}, p ; - \boldsymbol{x}^{\mathrm{ref}})\right\}
  \Omega_{\boldsymbol{x}^{\mathrm{ref}}},
  \label{e:PWNew023}
\end{equation}
where the directivities $\mathcal{D}_{1,2}$ are defined in \eqref{e:New14}-\eqref{e:New18}. 
Finally, as for the proof of Lemma \ref{cor:New02}, we can now consider the case where $\boldsymbol{x}$ moves to within $\mathcal{C}\setminus\mathbb{S}$. It is done by continuously deforming $\gamma_\infty$ on $\Xi$ such that it never hits $\Phi_+
(\boldsymbol{x})$ and $\Phi_- (\boldsymbol{x})$. The resulting contour is denoted $\tilde \Gamma_{\!\! \infty}$. The procedure is
illustrated in Figure \ref{fig:PWRep11}, right. The plane wave representation, valid for $\x\in\S^2\setminus(\mathbb{S}\cup(-\mathbb{S}))$ becomes
\begin{equation}
     u (\boldsymbol{x}) =\frac{\mathcal{B}_{\lambda}}{2} \int_{\tilde \gamma_{\infty}} \left\{\mathcal{D}_1 (p ; \boldsymbol{x}^{\mathrm{ref}}) w_2 (\boldsymbol{x},
  p ; - \boldsymbol{x}^{\mathrm{ref}}) - \mathcal{D}_2(p, \boldsymbol{x}^{\mathrm{ref}})
  w_1 (\boldsymbol{x}, p ;  -\boldsymbol{x}^{\mathrm{ref}})\right\}
  \Omega_{\boldsymbol{x}^{\mathrm{ref}}}.
  \label{e:PWNew023raph}
\end{equation}
Through this procedure, we find that the representation \eqref{e:PWNew023} can be extended to \eqref{e:PWNew023raph}
as long as neither $\Phi_+ (\boldsymbol{x})$, nor $\Phi_- (\boldsymbol{x})$ hit
$\Phi_+ (\mathbb{S})$ or $\Phi_-(\mathbb{S})$. If this were to be the case, the contour of integration would become pinched between $\Phi_{\pm}(\x)$ and $\Phi_{\pm}(\mathbb{S})$. 

\begin{figure}[h!]
    \centering
    \includegraphics[width=\linewidth]{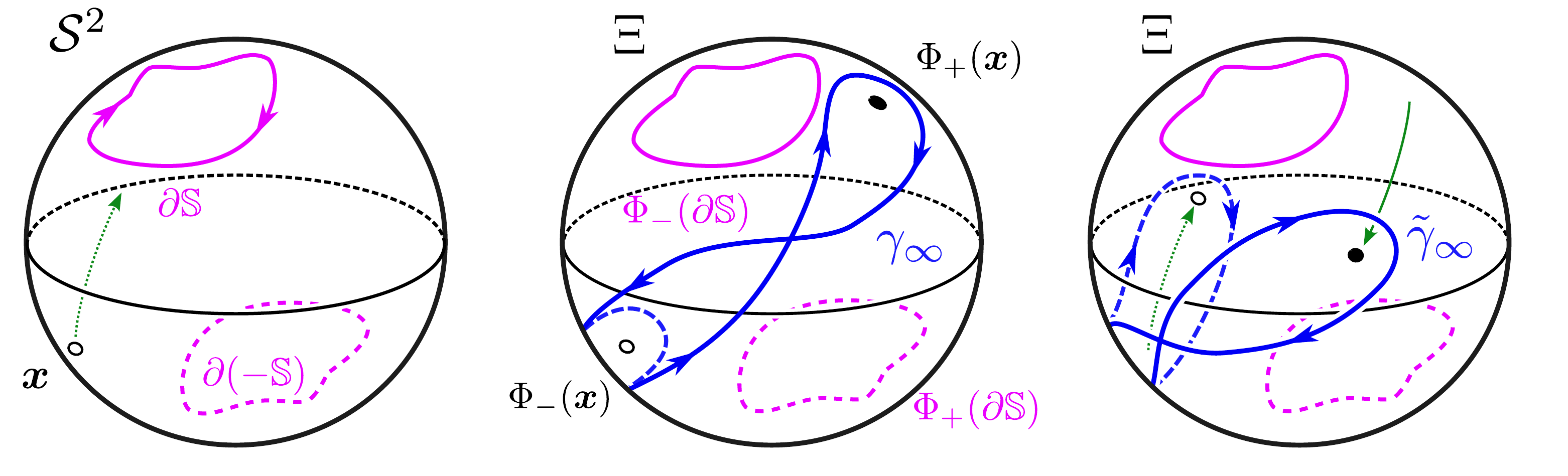}
    \caption{\small Illustration of why one may use the representation \eqref{e:NewPW23} within \eqref{e:New002} (left and middle), as well as contour deformation required for obtaining \eqref{e:PWNew023raph} (left and right). }
    \label{fig:PWRep11}
\end{figure}

Let us summarise our results in the following Theorem. 

    \begin{theorem}
        Let $u$ satisfy the Laplace-Beltrami boundary-value problem problem  \eqref{e:HLB}-\eqref{e:ID} formulated in Section~\ref{sec:pbformulationsphere}, let $\x^{\mathrm{ref}} \in \mathbb{S}$, let $w_{1,2}$ be as in \eqref{def:planewavesderaph}, and let $\mathcal{D}_{1,2}$ be given by \eqref{e:New14}-\eqref{e:New18}. Then, for $\x \in \mathcal{S}^2 \setminus \mathbb{S}$, $u$ can be written as either of the plane-wave representations given in \eqref{e:New16} and \eqref{e:New17}. For  $\x \in \mathcal{S}^2 \setminus (\mathbb{S} \cup - \mathbb{S})$, on the other hand, $u(\x)$ can be written as any of the three plane-wave representations given throughout \eqref{e:New16}, \eqref{e:New17}, and \eqref{e:PWNew023raph}.
    \end{theorem}

By definition, the plane-wave representation \eqref{e:PWNew023raph} is more restricted than the representations\eqref{e:New16}-\eqref{e:New17} in terms of range of validity.
    Still, there are some clear advantages to the representation \eqref{e:PWNew023raph}: Beyond the fact that it allows for more general locations of $\x^{\mathrm{ref}}$ compared to \eqref{e:New16}-\eqref{e:New17}, one can use either of the representations \eqref{e:New16}-\eqref{e:New17} to represent $u(\x)$ whenever $\x$ lies in the southern hemisphere $x_3 <0$, and use the representation \eqref{e:PWNew023raph} for the remaining observation directions $\x \in (\mathcal{S}^2 \setminus \mathbb{S})\cap \{x_3 \geq 0\}$. This is possibly advantageous, since, as illustrated in Figure \ref{fig:PWRep11}, the geometry of the $\tilde{\gamma}_{\infty}$ relative to $\Phi_{\pm}(\x)$ remains unchanged as $\x$ moves. So, akin to what is done in \citep{partAcomplexSphere} for the Green's function, one may  build a sliding-contours family consisting of $\tilde \gamma_{\infty}$ loops. That is, whenever $\x^{\mathrm{ref}} \in \mathbb{S}$, using the representation \eqref{e:PWNew023raph} alleviates the need for continuous `ad-hoc' contour deformation in order to extend the region of validity of $u(\x)$'s plane-wave representation.
  
  {The plane-wave representation formulae given in this section are the principal results of the present article. Our next goal is to show that the directivities $\mathcal{D}_1$ and $\mathcal{D}_2$ introduced in \eqref{e:New14} and \eqref{e:New18} are actually analogous to the directivity pattern $\mathcal{D}$ encountered in the planar setting described in Section \ref{sec:motivation}. To achieve this, we will need to make sense of what a `far-field' expansion of $u$ looks like, which requires a notion of infinity. Such infinity is absent on the compact sphere $\mathcal{S}^2$, but, as in \citep{partAcomplexSphere}, \textit{complexification} of $\mathcal{S}^2$ will introduce this sought-after (asymptotic) notion of infinity. This will allow us to analytically continue $u$ onto this complexification, in the spirit of \citep{AssierShanin2021AnalyticalCont}, and subsequently study its `far-field' asymptotics.}

\section{Analytical continuation to the complexified sphere} \label{sec:analytical-continuation}

\subsection{Complexification of the sphere and Green's function $G$}\label{sec:AC01}

We define the \textit{complexified sphere} as 
    \begin{align}
        \mathcal{S}^2_{\c} = \left\{\z \in \mathbb{C}^3:~ \z \cdot \z=1\right\}
    \end{align}
where, as everywhere below, we let 
    \begin{align}
        \z \cdot \z' = z_1 z'_1 + z_2 z_2' + z_3 z_3'
    \end{align}
 for any $\z = (z_1,z_2,z_3)$ and $\z' = (z_1',z_2',z_3')$ in $\mathbb{C}^3$. As discussed in \citep{partAcomplexSphere}, the complex sphere $\S^2_{\c}$ can be described by the coordinates $( \theta , \varphi )$
(both taking complex values) almost everywhere via formulae similar to \eqref{Ch02.eq.01}: 
\begin{equation}
z_1 = \sin \theta \, \cos \varphi , 
\quad 
z_2 = \sin \theta \, \sin \varphi , 
\quad 
z_3 = \cos \theta.
\label{e:f003}
\end{equation}
The coordinates $(\theta, \varphi)$ do not work regularly near the points~$\z = (0,0,\pm 1)$.
Neighbourhoods of these points
can be covered by, say, interchanging $z_2$ and $z_3$ in \eqref{e:f003}. 

Using these complexified coordinates, we can now also define the \textit{complexified Laplace--Beltrami equation} (CLBE), which was introduced in \citep[Section 2(d)]{partAcomplexSphere}. 
Let $U$ be an open domain in $\S^2_{\c}$.
Use the complex coordinates \eqref{e:f003} in~$U$. 
A function $\mathscr{u}(\z)$, $\z \in U$ is a solution to the CLBE if $\mathscr{u}(\theta , \varphi)$ is a holomorphic function of $(\theta, \varphi)$ in $U$, and 
it obeys there the equation having the same form as \eqref{e:HLB}:
\begin{equation}
\left( 
\tilde \Delta_{\c} 
 + \lambda (\lambda + 1)
\right) \mathscr{u} (\theta , \varphi) = 0,
\qquad 
\tilde \Delta_{\c} =  \ptl_{\theta}^2 + \frac{\cos \theta}{\sin \theta}\ptl_{\theta}+ 
\frac{1}{\sin^2 \theta}\ptl_{\varphi}^2.
\label{e:f007}
\end{equation}
Note that the CLBE \eqref{e:f007} is homogeneous, and the derivatives are understood in the complex holomorphic sense. In this case, we call $\tilde \Delta_{\c}$ the \textit{complexified Laplace--Beltrami operator} (CLBO). Even though we used complexified coordinates in order to define $\tilde \Delta_{\c}$, it can be shown that the definition of the CLBO does not depend on the choice of coordinate system, see \eqref{AppA.2} in Appendix \ref{App:Diff}.

 Within $\mathcal{S}^2_{\c}$, we will need to consider the following complex lines $L(\z; \bdeta)$ introduced in \citep[Section 2(c)]{partAcomplexSphere}. 

 \begin{definition}
   Let $\z \in \mathcal{S}^2_{\c}$ and $\bdeta \in \mathcal{N}$, with $\z \cdot \bdeta =0$. We define
    \begin{equation}
    L\left(\z; \bdeta\right)=\left \{\z^{\prime} \in \mathbb{C}^3 : \z^{\prime} = \z + c \bdeta \text{ for some }
        c \in \mathbb{C}\right\} \subset \mathcal{S}^2_{\c}, \label{e:bdd004}
    \end{equation}
   and call it the characteristic line passing through $\z$ with direction $\bdeta$. 
 \end{definition}

We next generalise the construction of the points $\Phi_{+}(\x)$ and $\Phi_-(\x)$, described in Section \ref{sec:singplanewaves} to complex $\z$. Consider the equation
\begin{align}
    \z \cdot \bdeta = 0, \quad \z \in \mathcal{S}^2_{\c}, ~ \text{ and } ~ \bdeta \in \mathcal{N} \label{e:6.6}
\end{align}
which defines the characteristic lines.
Solving \eqref{e:6.6} for $\bdeta$, using the parametrisation \eqref{e:f003}, we find
   \begin{equation}
	\boldsymbol{\eta}_{\pm}(\z) =  c \left(-i \cos \theta \, \cos \varphi \mp \sin \varphi
    \, , \,
    - i \cos \theta \, \sin \varphi \pm \cos \varphi
    \, , \, 
    i\sin \theta \right), 
    \qquad 
    c \in \mathbb{C} \setminus \{ 0 \}.
\label{e:ap004}
\end{equation}	
When $\z = \x \in \mathcal{S}^2$, we have $[\bdeta_+(\x)] = \Phi_+(\x) \in \Xi$ and $[\bdeta_-(\x)]=\Phi_-(\x)\in \Xi$, so $[\boldsymbol{\eta}_{\pm}(\z)]$ naturally extend the maps $\Phi_{\pm}$ to complex $\z$.

\begin{lemma} \label{lem:charac} The following statements are true \citep[Lemma 2.6]{partAcomplexSphere}:
\begin{enumerate}
\item
For each $\z \in \S^2_{\c}$, there are exactly two characteristic lines $L_{\pm}(\z)$ passing through $\z$, and these are given by $L_{\pm}(\z) = L(\z; \bdeta_{\pm}(\z))$. Consequently, for $\z \neq \z' \in \mathcal{S}^2_{\c}$, we have $L_+(\z) \cap L_+(\z') = \emptyset$ and $L_-(\z) \cap L_-(\z') = \emptyset$. 
\item The characteristic lines $L_{\pm}(\z)$ intersect $\mathcal{S}^2$ in a single point each. We denote these points by $\Psi_+(\z) = L_+(\z) \cap \mathcal{S}^2$ and  $\Psi_-(\z) = L_-(\z) \cap \mathcal{S}^2$. Explicit formulae for $\Psi_{\pm}(\z)$ can be obtained from \citep[Appendix A]{partAcomplexSphere}. 
We will comment in Remark \ref{rem6.3} on why we choose notations similar to the functions $\Psi_{\pm}(p)$, which were defined in Section \ref{sec:singplanewaves}.
\item
Let  $\z'\in \S^2_{\c}$. Then,
\begin{equation}
L_+ ( \z' ) \cup L_- ( \z' ) = 
\{ \z \in \S^2_{\c} 
\, \, : \, \,
\z \cdot \z' = 1
\}.
\label{e:f006c}
\end{equation}
\end{enumerate}
\end{lemma}

\begin{remark}\label{rem6.3}
    As explained in \citep[Section 2(b)]{partAcomplexSphere}, by working within the complex projective space $\CP^3$, we see that $\mathcal{S}^2_{\c}$ is embeddable into the quadric surface 
    \[\hat{\mathcal{S}}^2_{\c} = \{\class{z_0:z_1:z_2:z_3} \in \CP^3: z_1^2+z_2^2+z_3^2=z_0^2\}\] via the embedding
    $\iota(z_1,z_2,z_3) = \class{1:z_1:z_2:z_3}$.  Thus, through this embedding, we may view $\Psi_{\pm}(\z)$ as maps defined on $\hat{\mathcal{S}}^2_{\c}$. This is described in detail in \citep{partAcomplexSphere}. 
    We also observe that, by definition, $\Xi \subset \hat{\mathcal{S}}^2_{\c}$. What is important for the present article is that, through this complex-projective viewpoint, the maps $\Psi_{\pm}(\z)$ 
    can be seen as continuations of the maps $\Psi_{\pm}: \Xi \to \mathcal{S}^2$ introduced in Section \ref{sec:singplanewaves}, which justifies the similar terminology. 
\end{remark}

The characteristic lines will play an important role when analytically continuing the scattered field $u$ to $\S_{\c}^2$ as we will see in Section \ref{sec:AC02}. This is because the characteristic lines dictate the propagation of singularities of $G$. This is made precise in the following Lemma.

 \begin{lemma}\label{lem:G}
     The free-space Green's function $G(\x;\x_0)$ admits an analytic continuation, in either argument, onto $\mathcal{S}^2_{\c}$. This analytic continuation is defined via 
       \begin{equation}
         G(\z; \z_0) = \frac{1}{4 \sin (\pi \lambda)} P_{\lambda} (-\z \cdot \z_0), ~ \text{ for } ~ \z, \z_0 \in \S^2_{\c}. \label{e:f010}
        \end{equation}
      For fixed $\z_0 \in \mathcal{S}^2_{\c}$, the \textit{complexified} Green's function $G(\z;\z_0)$ is analytic, and generally multivalued, within $\S^2_{\c} \setminus (T(\z_0)\cup T(-\z_0))$, where 
    \begin{align}
        T(\z_0) =  L_+(\z_0) \cup L_-(\z_0) = \{ \z \in \S^2_{\c} 
\, \, : \, \,
\z \cdot \z_0 = 1
\}. \label{e:6.8}
    \end{align}
    In \eqref{e:f010}, $P_{\lambda}$ is understood as the (multivalued) analytical continuation of the Legendre function to $\mathbb{C} \setminus \{ 1,-1 \}$, as described in \citep[Section 3]{Bateman1953}. 
    The analytical continuation $G(\z;\z_0)$ solves the CLBE \eqref{e:f007} on $\mathcal{S}^2_{\c} \setminus (T(\z_0) \cup T(-\z_0))$.  
\end{lemma}

Let us note that similar analyticity statements are true for the normal derivative of $G$. Let $\n$ denote the outwards pointing normal vector on $\partial \mathbb{S}$, as shown in Figure \ref{fig:formulation-fig}, right. Using the formula $\partial_{\n}G(\z;\y) = \n \cdot \nabla G(\z;\y)$ and the fact that  
    \begin{align}
        \frac{d}{d\tau}P_{\lambda}(\tau)=\frac{\lambda(P_{\lambda-1}(\tau) - \tau P_{\lambda}(\tau))}{1-\tau^2}, \label{e:normalD}
    \end{align}
 we find that, for fixed $\y \in\S^2$, whenever $G(\z;\y)$ is analytic in a region $U \subset \mathcal{S}^2_{\c}\setminus (T(\y)\cup T(-\y))$, the normal derivative $\partial_{\n}G(\z;\y)$ is analytic in $U$ too. Moreover, $\partial_{\n} G(\z;\y)$ solves the CLBE in $\z$. 

It is also possible to analytically continue $\partial_{\n}G(\z;\y)$ in the $\y$ variable. This is discussed in more detail in Section \ref{sec:AC03} and the Appendix \ref{app:diffSph}.
The complexification of the inhomogeneous Equation \eqref{eq.Ch02--GreensFunctionGPDE} is not straightforward, however, and will be discussed in more detail at the end of Section \ref{sec:AC03}.

\subsection{Analytic continuation of the scattered field $u$}\label{sec:AC02}

Akin to how the Green's function $G(\x;\x_0)$ admits an analytical continuation $G(\z;\z_0)$ onto $\mathcal{S}^2_{\c} \setminus (T(\z_0) \cup T(-\z_0))$, we will now show that the scattered field $u(\x)$, introduced in Section \ref{sec:pbformulationsphere} and solution to the boundary value problem \eqref{e:HLB}-\eqref{e:ID}, can be analytically continued onto $\mathcal{S}^2_{\c} \setminus (T(\mathbb{S}) \cup T(-\mathbb{S}))$. Here, we have used the notation 
    \begin{align}
        T(X) = L_+(X) \cup L_-(X), ~ \text{ with } ~ L_{\pm}(X) =\bigcup_{\x \in X} L_{\pm}(\x)
    \end{align}
 for $X = \mathbb{S}$ and $X= - \mathbb{S}$.   
 Before we proceed with the formulation of the Theorem, let us note that whether we may analytically continue $u$ onto the set $T(-\mathbb{S})$ depends on the choice of path for analytical continuation. We will discuss this in more detail after the formulation, and proof, of the following Theorem.

 \begin{theorem}\label{thm:6.5}
     Let $\x \in \mathcal{S}^2 \setminus (\mathbb{S} \cup - \mathbb{S})$, $\z \in \mathcal{S}^2_{\c} \setminus (T(\mathbb{S}) \cup T(-\mathbb{S}))$, and consider a path 
        \[\cc:[0,1] \to \mathcal{S}^2_{\c} \setminus (T(\mathbb{S}) \cup T(-\mathbb{S}))\] 
    with $\cc(0)=\x$ and $\cc(1) = \z$.\footnote{Since $\mathbb{S} \cap -\mathbb{S} = \emptyset$, and $\mathbb{S}$ is simply connected, such path always exists, as can be shown using Segre embedding \citep[Chapter 4]{griffiths1984principles}.}  Let us denote by $G(\cc(t);\y)$ and $ \partial_{\n} G(\cc(t);\y)$ the analytical continuation of $G(\x;\y)$ and $ \partial_{\n} G(\x;\y)$ along $\cc$. Then, the formula 
        \begin{align}
            u(\cc(t)) = \int_{\partial \mathbb{S}} [u(\y) \partial_{\n} G(\cc(t);\y) - G(\cc(t);\y) \partial_{\n} u(\y)] d\ell \label{e:6.9}
        \end{align}
       provides the analytical continuation of $u$ along $\cc$. We will also write $u(\z)$ instead of $u(\cc(1))$ whenever the choice of path taken for analytic continuation is either clear, or not relevant. The function $u(\z)$ solves the CLBE \eqref{e:f007} in $\mathcal{S}^2_{\c} \setminus (T(\mathbb{S}) \cup T(-\mathbb{S}))$.
 \end{theorem}

    \begin{proof}
        Consider the formula \eqref{e:6.9}. For $t=0$, this formula is simply Green's third identity \eqref{e:New002}. 
        Now,
        for $t\in [0,1]$, the singularities of $G(\cc(t);\y)$ on $\mathcal{S}^2$ are given by $\pm \Psi_+(\cc(t))$ and $\pm \Psi_-(\cc(t))$, where the singularities located at $-\Psi_+(\cc(t))$ and $- \Psi_-(\cc(t))$ need to be taken into account due to the fact that $G(\z;\y)$ is generally singular for $\z \in T(-\y)$.  The same statement is true for the normal derivative $\partial_{\n} G(\cc(t);\y)$, taken with respect to $\y$.
        Now, $\cc$ never hits $T(\mathbb{S})$ or $T(-\mathbb{S})$. Thus, according to points 1 and 2 of Lemma \ref{lem:charac}, these singularities of $G(\cc(t);\y)$ will never hit $\partial \mathbb{S}$, the contour of integration in \eqref{e:6.9}. Indeed, if, say, $\Psi_+(\cc(t))$ hits $\partial \mathbb{S}$ at $t=t^{\star}$, then necessarily $\cc(t^{\star}) \in L_+(\mathbb{S})\subset T(\mathbb{S})$, which is a contradiction. Therefore, we find that \eqref{e:6.9} establishes the sought-after analytical continuation of $u(\x)$.  The evolution of the singularities of $G$ on $\mathcal{S}^2$ as $\z$ moves along $\cc$ is schematically illustrated in Figure \ref{fig:AC01}.
        
        The statement that $u(\z)$ solves the CLBE in $\mathcal{S}^2_{\c} \setminus (T(\mathbb{S}) \cup T(-\mathbb{S}))$ now follows from the fact that $G(\z;\y)$ and $\partial_{\n}G(\z;\y)$ both solve the CLBE in the variable $\z \in \mathcal{S}^2_{\c} \setminus (T(\y) \cup T(-\y))$, as was discussed in the previous Section \ref{sec:AC01}. 
    \end{proof}

        \begin{figure}[h!]
            \centering
            \includegraphics[width=\linewidth]{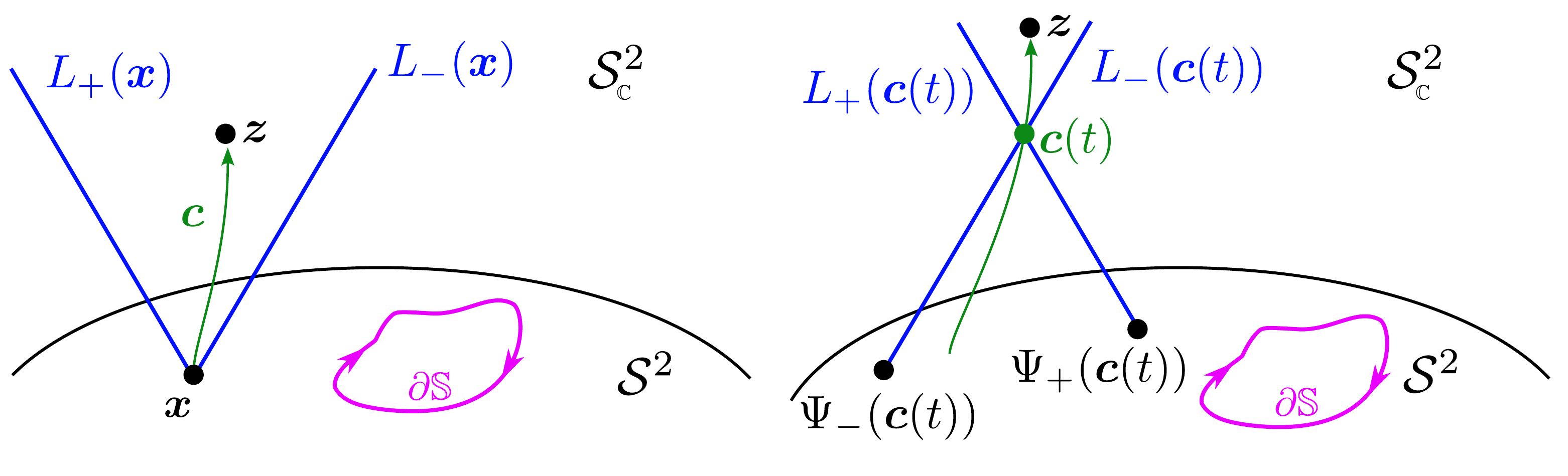}
            \caption{\small Diagrammatic illustration of the evolution of the singularities $\Psi_{\pm}(\cc(t))$ on $\mathcal{S}^2$ as we analytically continue along a path $\cc$ connecting $\x$ to $\z$. 
            Note that each characteristic line $L_{\pm}(\z)$ is of real dimension 2, while the ambient $\mathcal{S}^2_{\c}$ is of real dimension 4, so we necessarily need to suppress one dimension for visualisation.}
            \label{fig:AC01}
        \end{figure}


   The set $T(-\mathbb{S})$ does not always need to be excluded from the domain of analytical continuation: the function $u(\x)$ defined in Section \ref{sec:pbformulationsphere} is regular on $\mathcal{S}^2\setminus \mathbb{S}$. So, at least in a complex neighbourhood $U \subset \mathcal{S}^2_{\c}$ of $-\mathbb{S}$, there exists an analytical continuation $u(\z)$ of $u(\x)$. Indeed, whether we may analytically continue $u$ onto points $\z \in T(-\mathbb{S})$ depends strongly on the path $\cc$ chosen for analytical continuation. For some fixed $\y\in\partial\mathbb{S}$, we will say that $G(\cc(t);\y)$ belongs to the \textit{physical branch} of $G$ whenever the scalar path $-\cc(t)\cdot\y$ does not turn around $-1$, which is a branch point of $P_\lambda$. In this way, we ensure that $-\cc(t)\cdot\y$ remains on the principal branch of $P_\lambda$, which is only singular at $-1$, and regular at $+1$, unlike the other branches of $P_\lambda$, which are generally singular at $+1$, see \cite{Bateman1953}. This gives us a simple criterion for $u$: the analytical continuation $u(\z)$ is regular on $T(-\mathbb{S})$ whenever $\cc(t)$ is such that $-\cc(t)\cdot\y$ remains on the principal branch of $P_\lambda$  for all $\y\in\partial\mathbb{S}$. This is basically due to the fact that the singularities of the physical branch of $G$ never hit the contour of integration $\partial \mathbb{S}$ in this process, which is in general not true for arbitrary branches of $G$.
   
   Note that one can also extend the notion of physical branch of $G$ for complex second argument. Indeed if $G(\z;\z')$ is a continuation from $G(\x;\x')$ where $\x\to\z$ and $\x'\to\z'$ along some paths $\cc(t)$ and $\dd(t)$ respectively, then provided $-\cc(t)\cdot\dd(t)$ remains on the principal branch of $P_\lambda$, we say that $G(\z;\z')$ belongs to the physical branch of $G$.  In this case, as a function of $\z$, $G(\z;\z')$ is regular on $T(-\z')$.}

\subsection{Complexified Green's third identity}\label{sec:AC03}

So far within Section \ref{sec:analytical-continuation}, we have discussed the analyticity properties of $u$ within $\mathcal{S}^2_{\c}$. This will be instrumental in order to make sense of what a `far-field' expansion of $u$ looks like, as we will discuss in Section \ref{sec:farfield}. Before we proceed with this, however, let us note that there are still two important concepts, available on $\mathcal{S}^2$, that have been left out from our complexification discussion. These two concepts are Green's third identity \eqref{e:Greens3rdSphere} for a contour of integration lying on $\mathcal{S}^2_{\c} \setminus \mathcal{S}^2$, 
and the inhomogenous Laplace-Beltrami equation \eqref{eq.Ch02--GreensFunctionGPDE}. In this section, we will discuss how both of these concepts can be complexified.


    To achieve this, we will introduce the following $(\omega_1,\omega_2)$-coordinate system, which is centred at a given point $\z_0 \in \mathcal{S}^2_{\c}$  
    and parametrises a sufficiently small neighbourhood $U \subset \mathcal{S}^2_{\c}$ of $\z_0$. We will make precise the `sufficiently small' notion in Remark \ref{rem:w1w2} below. To construct this coordinate system, fix $\z_0 \in \mathcal{S}^2_{\c}$, choose any two representatives $\bdeta_-(\z_0)$ and $\bdeta_+(\z_0)$ amongst the family of vectors given in \eqref{e:ap004}, and note that the three vectors $\z_0, ~ \bdeta_-(\z_0),$ and $\bdeta_+(\z_0)$ span $\mathbb{C}^3$. So, in particular, we may represent \textit{any} $\z \in \mathcal{S}^2_{\c}$ as 
        \begin{align}
            \z = a \z_0 + b \bdeta_-(\z_0) + c \bdeta_+(\z_0), ~ \text{ for } ~ a,b,c \in \mathbb{C}.
        \end{align}
      Let us now assume that $\z$ is `close' to $\z_0$ (see Remark \ref{rem:w1w2}).

Using the algebraic relations $\z \cdot \z = \z_0 \cdot \z_0 = 1,$ $\bdeta_{\pm}(\z_0) \cdot \bdeta_{\pm}(\z_0) = \bdeta_{\pm}(\z_0) \cdot \z_0 =0$, we find that     
    \begin{equation}
    \z = (1+ \alpha) \z_0 + \omega_1 \boldsymbol{\eta}_-(\z_0) + \omega_2 \boldsymbol{\eta}_+(\z_0)
    \label{e:AS019.1}
    \end{equation}
    for  
\begin{equation}
\alpha = -1 + \sqrt{1 - 2 \omega_1 \omega_2 \boldsymbol{\eta}_+(\z_0) \cdot \boldsymbol{\eta}_-(\z_0)} 
\label{e:AS019.2}
\end{equation}
and 
\begin{equation}
\omega_1 = \frac{\z \cdot \boldsymbol{\eta}_+(\z_0)}{\boldsymbol{\eta}_+(\z_0) \cdot \boldsymbol{\eta}_-(\z_0)},
\qquad 
\omega_2 = \frac{\z \cdot \boldsymbol{\eta}_-(\z_0)}{\boldsymbol{\eta}_+(\z_0) \cdot \boldsymbol{\eta}_-(\z_0)}.
\label{e:AS019.3}
\end{equation}

\begin{remark}\label{rem:w1w2}
    When using the $(\omega_1,\omega_2)$ coordinate system, we assume that $\z$ is close to $\z_0$ in the sense that $\omega_1$ and $\omega_2$ are sufficiently small to ensure $|2 \omega_1 \omega_2 \bdeta_+(\z_0) \cdot\bdeta_-(\z_0)| <1$.  This guarantees that we may take a single-valued branch of the square root to represent $\alpha$ according to \ref{e:AS019.2}. This branch is taken according to $\sqrt{1}=1$, so that $\z \to \z_0$ as $\omega_{1,2} \to 0$. 
    When every $\z$ in a neighbourhood $U$ of $\z_0$ can be parametrised by  $(\omega_1,\omega_2)$ satisfying these conditions, $U$ will be referred to as `sufficiently small', and only such neighbourhoods $U$ will be considered within this Section \ref{sec:AC03}.
\end{remark}

The pair $(\omega_1,\omega_2)$ yields a local complex coordinate system near $\z_0$, such that $L_+(\z_0)$ corresponds, locally, to $\{\omega_1 \equiv 0\}$ and $L_-(\z_0)$ corresponds to $\{\omega_2 \equiv 0\}$, as illustrated schematically in Figure \ref{fig:AC02}. We will also write 
    \[(\omega_1,\omega_2) =  (\omega_1(\z_0),\omega_2(\z_0))\] 
to highlight the fact that this coordinate system is centred at $\z_0$. 

Next, consider the loop 
\begin{align}
      \tilde \Gamma_{\!\! \infty}(\z_0;\rho) \subset U \setminus (L_+(\z_0) \cup L_-(\z_0))
\end{align}
   with local $(\omega_1,\omega_2)$ coordinates defined according to 
  \begin{align}
      (\omega_1,\omega_2)=(\rho e^{-i \theta}, \rho e^{i \theta}) ~ \text{ for } ~ 0 \leq \theta \leq 2\pi ~\text{ and } ~\rho \in \mathbb{R} ~ \text{ small and positive.} \label{e:6.18}
  \end{align} 
    Whenever the precise choice of $\rho$ does not matter, or is clear from context, we will also allow ourselves to write 
    \begin{align}
        \tilde \Gamma_{\!\! \infty}(\z_0;\rho) \equiv \tilde \Gamma_{\!\! \infty}(\z_0).
    \end{align}
 This loop is schematically illustrated in Figure \ref{fig:AC02}. The rationale for depicting $\tilde \Gamma_{\!\! \infty}(\z_0)$ as a figure-of-eight loop is given in Remark \ref{rem:7.3}, and indicative of the loop's different orientations in the complex $\omega_1$ and $\omega_2$ planes (see Figure \ref{fig:AC02}, middle and right).
    \begin{figure}[h!]
        \centering
        \includegraphics[width=\linewidth]{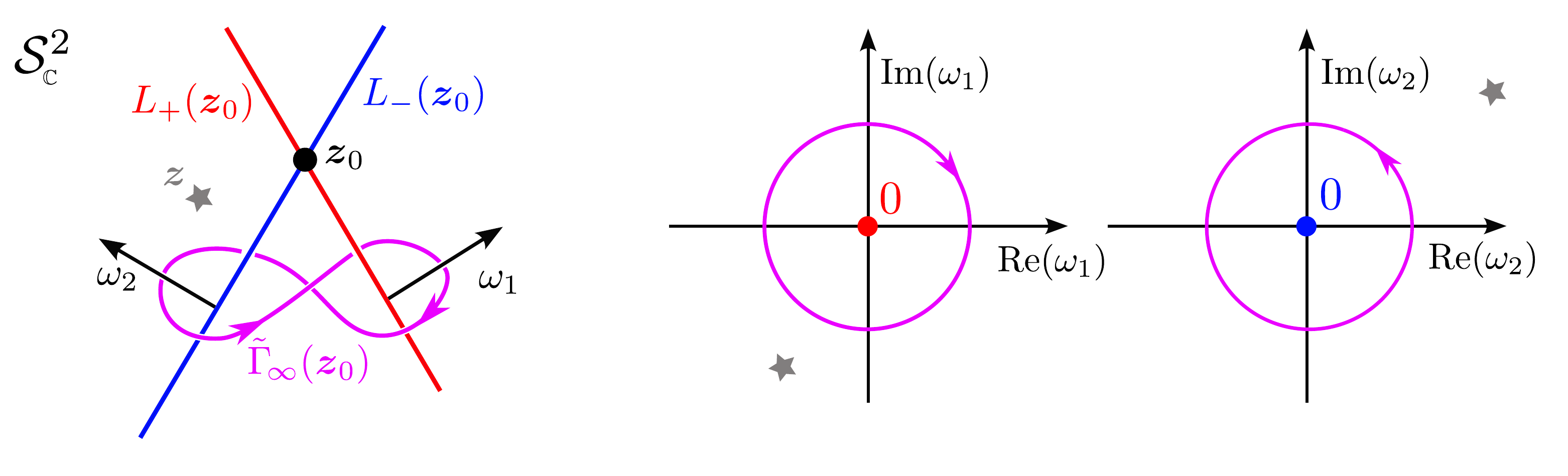}
        \caption{\small Left: Diagrammatic illustration of the $(\omega_1,\omega_2)$ coordinate system centred at $\z_0$, and loop $\tilde \Gamma_{\!\! \infty}(\z_0)$. Middle and Right: complex $\omega_1$ and $\omega_2$ planes. Here, we also see the projections of $\tilde \Gamma_{\!\! \infty}(\z_0)$ (loops) and $\z$ (stars) onto the complex $\omega_1$ and $\omega_2$ planes, respectively. The origins in the complex $\omega_1$ and $\omega_2$ planes correspond to the characteristics $L_+(\z_0)$ and $L_-(\z_0)$, respectively.}
        \label{fig:AC02}
    \end{figure}

Now, consider Green's third identity \eqref{e:Greens3rdSphere} on $\mathcal{S}^2$. Note that for any smooth function $f$, we have 
    \begin{align}
	\partial_{\n}f(\y)d\ell = \sin(\theta)\partial_{\theta}f d \varphi - \frac{1}{\sin(\theta)} \partial_{\varphi}f d \theta, \label{eq.Ch02--RealGradient1Form}
    \end{align} 
    and the right-hand-side of \eqref{eq.Ch02--RealGradient1Form} \textit{also} makes sense for complexified angular variables $(\theta,\varphi)$. Note that we can use \eqref{eq.Ch02--RealGradient1Form} in combination with \eqref{e:normalD} to analytically continue the normal derivative $\partial_{\n}G(\z;\y)$ in the $\y$-variable.

    Second, consider two smooth functions $f(\z')$ and $g(\z')$. 
    Assume that $\z'$ is parametrised according to \eqref{e:f003} upon symbolically replacing $(\theta,\varphi)$ by $(\theta',\varphi')$. We choose to work with $\z'$ instead of $\z$ to avoid confusion later on, regarding the variable of integration, which will be taken to be $\z'$. Let 
         \begin{align*}
          \G[f(\z'), g(\z')] = & \phantom{-}
            \sin(\theta') \left[ f(\z')\partial_{\theta'}g(\z') - g(\z')\partial_{\theta'}f(\z') \right] d \varphi'  \\
            &- \frac{1}{\sin(\theta')}\left[f(\z')\partial_{\varphi'}g(\z') - g(\z') \partial_{\varphi'}f(\z') \right] d\theta'. \label{e:6.24} \numberthis
        \end{align*}
    The operator $\G$ 
    has the following properties, which are proved in Appendix \ref{app:diffSph}.
        \begin{enumerate}
            \item Using differential form formalism, it can be shown that the definition of $\G$ is independent of the choice of coordinate system;
            \item Whenever $f$ and $g$ both solve the CLBE in some region, then $\G[f(\z'),g(\z')]$ is a \textit{closed holomorphic differential form} in that same region. 
        \end{enumerate}
   Now, using \eqref{eq.Ch02--RealGradient1Form} and \eqref{e:6.24}, we can rewrite Green's third identity \eqref{e:Greens3rdSphere} on $\mathcal{S}^2$ as 
    \begin{align}
        u(\x) = \int_{\tilde{\Gamma}} \G[u(\x'); G(\x;\x')], \label{e:6.25}
    \end{align}
    where we have replaced the variable $\y$ of integration by $\x'$. This change is only symbolic to make the transition to the formulation of the complexified Green's third identity \eqref{e:6.019} more clear. Before we can prove the complexified Green's third identity, we require the following lemma.


\begin{lemma}\label{lem:6.7}
        Let $\x \in \mathcal{S}^2$, and consider a small loop $\tilde{\Gamma} \subset \mathcal{S}^2$ localised about $\x$ as illustrated in Figure  \ref{fig:formulation-fig}, right. Let $U \subset \mathcal{S}^2_{\c}$ be a sufficiently small complex neighbourhood of $\x$, in the sense of Remark \ref{rem:w1w2}. Then, the loop $\tilde{\Gamma}$ is homotopic to $\tilde{\Gamma}_{\!\!\infty}(\x) \subset U \setminus (L_+(\x) \cup L_-(\x))$.
    \end{lemma}

    \begin{proof}
        The proof is purely topological. First, observe that $\tilde{\Gamma} \subset U \setminus (L_+(\x) \cup L_-(\x))$. Next, the neighbourhood $U \setminus (L_+(\x) \cup L_-(\x))$ retracts onto the 2-torus $\mathbb{T}^2 = \mathcal{S}^1 \times \mathcal{S}^1$. Both, $\tilde \Gamma_{\!\! \infty}(\x)$ and $\tilde{\Gamma}$ wind around each characteristic line $L_+(\x)$ and $L_-(\x)$ exactly once. By taking an explicit parametrisation of $\tilde{\Gamma}$, and upon computing the projections of $\tilde{\Gamma}$ onto the $\omega_1$ and $\omega_2$ coordinates, it can be checked explicitly that $\tilde{\Gamma}$ bypasses $L_{\pm}(\x)$ in the same direction as $\tilde{\Gamma}_{\!\! \infty}(\x)$. Using standard facts concerning the fundamental group $\pi_1(\mathbb{T}^2) \cong \mathbb{Z}^2$ of the 2-torus, as explained in \citep{Hatcher2000}, we conclude that $\tilde{\Gamma}$ and $\tilde \Gamma_{\!\! \infty}(\x)$ are homotopic in $U \setminus (L_+(\x) \cup L_-(\x))$. 
    \end{proof}


\begin{lemma}\label{lem:GreensComplex}
    Let $\x \in \mathcal{S}^2 \setminus (\mathbb{S} \cup  -\mathbb{S})$ and $\z \in \mathcal{S}^2_{\c} \setminus (T(\mathbb{S}) \cup  T(-\mathbb{S}))$. Let $\cc$ be a path in $\mathcal{S}^2_{\c} \setminus (T(\mathbb{S}) \cup  T(-\mathbb{S}))$ connecting $\x$ to $\z$.
    Then, 
        \begin{align}
            u(\cc(t)) = \int_{\tilde \Gamma_{\!\! \infty}(\cc(t))} \G[u(\z'), G(\cc(t);\z')]. \label{e:6.19}
        \end{align}
  We thus have the following \emph{complexified Green's third identity}
     \begin{align}
            u(\z) = \int_{\tilde \Gamma_{\!\! \infty}(\z)} \G[u(\z'), G(\z;\z')], ~ \text{ for } ~ \z \in \S^2_{\c} \setminus (T(\mathbb{S}) \cup T(-\mathbb{S})).\label{e:6.019}
        \end{align}
   When writing \eqref{e:6.019}, the branches are taken according to \eqref{e:6.19}.    
\end{lemma}

\begin{proof}
       For each $t$, consider the coordinate system $(\omega_1(\cc(t)),\omega_2(\cc(t))$, centred at $\cc(t)$, and corresponding loop $\tilde \Gamma_{\!\! \infty}(\cc(t)) \subset U_t\setminus T(\cc(t))$ where the supscript ${}_{t}$ indicates that $U$ varies with $t$. For example, $U_t$ could be a small polydisc centred at $\cc(t)$. Now, since $\G[u(\z'),G(\x;\z')]$ is a closed differential form (Appendix \ref{App:Diff}), we find, according to Lemma \ref{lem:6.7}, Equation \eqref{e:6.25}, and the generalised theorem of Stokes, that  
            \begin{align}
                u(\x) = \int_{\tilde \Gamma_{\!\! \infty}(\cc(0))} \G[u(\z'),G(\x;\z')].
            \end{align}
        That is, we may continuously deform the contour of integration from $\tilde\Gamma$ to ${\tilde\Gamma_{\!\!\infty}(\x) = \tilde\Gamma_{\!\!\infty}(\cc(0))}$ without changing the value of the integral, and we have replaced the variable $\x'$ of integration in \eqref{e:6.25} by $\z'$ to emphasize the fact that $\z' \in \tilde\Gamma_{\!\!\infty}(\cc(0))$ is, generally, not real.

        Now, repeating the proof of Lemma \ref{lem:6.7} verbatim, it can be shown that for $\z_1$ sufficiently close to $\z_2$, the contours $\tilde \Gamma_{\infty}(\z_1)$ and $\tilde \Gamma_{\infty}(\z_2)$ are homotopic and consequently, by the generalised theorem of Stokes, and analyticity of integrals depending on a parameter, the formula 
            \begin{align}
                u(\cc(t)) = \int_{\tilde \Gamma_{\!\! \infty}(\cc(t))} \G[u(\z'),G(\cc(t);\z')]
            \end{align}
        defines the analytical continuation of $u$ along $\cc$.
    \end{proof}

\subsection{Towards a characterisation of a complexified source}

Let us now discuss the `complexification' of the inhomogeneous LBE \eqref{eq.Ch02--GreensFunctionGPDE}. Until the end of this section, we will work solely with $G$ for a source located at $\x_0\in\S^2$, so we do not need to consider the scatterer. For simplicity, we will consider a spherical coordinate system $(\theta, \varphi)$ centred at $\x_0$ (or we could equivalently assume without loss of generality that that $\x_0=\xNP$). We are particularly interested in finding a way to characterise the source term after complexification. In the real case this is done through the forcing term $\delta(\x-\x_0)$. It is important to remember that to find the Green's function $G(\x;\x_0)$ in terms of $P_\lambda$, we made two assumptions. The first is that the source is isotropic (i.e.\ it radiates in the same way in all direction) so $G(\x;\x_0)$ should be independent of $\varphi$. The second is that we imposed $\lambda\notin\mathbb{Z}$ so that no non-trivial solution of the LBE on the whole sphere exists. 

Using the asymptotic behaviour of $P_\lambda$ near $-1$, as shown in \cite{partAcomplexSphere}, we find that
\begin{align}
    G(\x;\x_0) &= \frac{1}{2\pi} \log(\theta)+\mathcal{O}(1) \text{ as } \theta\to0 \text{ (i.e.\ as } \x\to\x_0\text{)}.
\label{eq:smallthetaasymptG}
\end{align}
Integrating the LBE \eqref{eq.Ch02--GreensFunctionGPDE} on a small disc of radius $\rho$ centred at $\x_0$ with boundary $\tilde \Gamma(\x_0;\rho)$, using the divergence theorem, the main property of the Dirac delta function, and \eqref{eq:smallthetaasymptG}, we find that
\begin{align}
    \lim_{\rho \to 0} \int_{\tilde{\Gamma}(\x_0;\rho)} \partial_{\n} G(\x;\x_0) d \ell =1,  \label{e:6.30}
\end{align}
where the integration and the normal derivative are performed with respect to $\x$. With the notations introduced earlier in this section, we can rewrite it as
\begin{align}
    \lim_{\rho \to 0} \int_{\tilde{\Gamma}(\x_0;\rho)} \G[1,G(\x;\x_0)]=1,  \label{eq:newformRealCharacG}
\end{align}
where again $\x$ is the variable of integration. We can then think of \eqref{eq:newformRealCharacG} as a characterisation of the source. 

Our aim is to obtain a similar result that could characterise the complexification of the source in the absence of a clear way of complexifying the $\delta(\x-\x_0)$ forcing term.
Consider therefore the analytical continuation of $G(\z,\z_0)$ for some $\z,\z_0 \in \S_{\c}^2$ obtained along some path chosen such that we stay on the physical branch of $G$.

Using the definition \eqref{e:AS019.1} of the $(\omega_1,\omega_2)$ coordinate system centred at $\z_0$ in combination with the explicit formula \eqref{e:f010} and the asymptotic properties of $P_{\lambda}$, we find
    \begin{align}
    G(\z; \z_0) = \frac{1}{4\pi} \left(
    \log \omega_1 + \log \omega_2
    \right) + \mathcal{O}(1), ~ \text{ as } ~ \omega_{1,2} \to 0. \text{ (i.e.\ as } \z\to\z_0\text{)}.
    \label{e:AS033new}
    \end{align}
Using \eqref{e:AS033new} and the definition \eqref{e:6.18} of the contour $\tilde \Gamma_\infty$, direct integration leads to
    \begin{align}
        \lim_{\rho\to0} \int_{\tilde \Gamma_{\!\! \infty}(\z_0;\rho)} \G[1,G(\z;\z_0)]=1, \label{e:6.31}
    \end{align}
where the integration is with respect to $\z$. This can be thought off as a characterisation of the `complex source' $\z_0$. Of course $G(\z;\z_0)$ should also satisfy the CLBE
\begin{align}
    \left(\tilde \Delta_{\c} +\lambda(\lambda+1)\right) G(\z;\z_0)&=0 \text{ on } \S^2_{\c}\setminus T(\z_0).
\end{align}



\section{Far-field asymptotics of the scattered wave-field} \label{sec:farfield}

Having established analyticity of $u$ in $\mathcal{S}^2_{\c} \setminus (T(\mathbb{S}) \cup T(-\mathbb{S}))$, we can now legitimately try to understand the `far-field behaviour' of $u(\z)$ as $|\z| \to \infty$ within $\mathcal{S}^2_{\c} \setminus (T(\mathbb{S}) \cup T(-\mathbb{S}))$. 


To this end, fix a reference point $\x^{\rm ref} \in \mathbb{S}$ and introduce the complex  coordinate system $(\epsilon, \beta)$ as follows. We will explain the need for introducing $\x^{\mathrm{ref}}$ in Remark \ref{rem7.1} below.

Let $A_{\x^{\mathrm{ref}}}\in \SO(3)$ denote a real Euler rotation such that 
    \begin{align}
        A_{\x^{\mathrm{ref}}}(0,0,1) = \x^{\mathrm{ref}},
    \end{align}
and parametrise $\z \in \mathcal{S}^2_{\c} \setminus (T(\mathbb{S}) \cup T(-\mathbb{S}))$  according to
\begin{equation}
    \z = \epsilon^{-1}  A_{\x^{\mathrm{ref}}} \left( \cos \beta \, , \, \sin \beta \, , \, i \sqrt{1 - \epsilon^{2}} \right). \label{e:bdd002}
\end{equation}
As $\epsilon\to0$, the point $\z$ will be such that $|\z|\to\infty$, and we can actually associate this limit with a specific point of $\Xi$ given by $[\bdeta]$ for $\bdeta=A_{\x^{\mathrm{ref}}} \left( \cos \beta \, , \, \sin \beta \, , \, i  \right)$. This is discussed in more detail in \citep{partAcomplexSphere}. Here, we only note that the point $[\bdeta] \in \Xi$ can be interpreted as a direction along which we are going towards infinity. More precisely, as will be discussed after the proof of Theorem \ref{thm:7.1}, the point $[\bdeta]$ characterises a complex \textit{observation direction} relative to the reference point $\x^{\rm ref}$.


\begin{remark}\label{rem7.1}
    Using the theory of complex projective spaces, it can be shown that the singularities of the coordinate system \eqref{e:bdd002} located at $\beta = \pm i \infty$ describe the characteristic lines $L_{\pm}(\x^{\mathrm{ref}}) \subset T(\mathbb{S})$.
    So, rotating $(0,0,1)$ onto $\x^{\mathrm{ref}} \in \mathbb{S}$ guarantees that we may use the coordinate system \eqref{e:bdd002} to parametrise any $\z \in \mathcal{S}^2_{\c} \setminus (T(\mathbb{S}) \cup T(-\mathbb{S}))$ with $|\z|$ large: the corresponding coordinate singularities have been `removed' from $\mathcal{S}^2_{\c}$. 
\end{remark}

To emphasize the dependence of this parametrisation on the reference point, we shall write 
    \begin{align}
        u(\z) = u(\epsilon,\beta; \x^{\mathrm{ref}}) \label{e:7.3}
    \end{align}
 whenever $\z$ is parametrised according to \eqref{e:bdd002}. Since we let $|\z| \to \infty$, we consider only small $|\epsilon|<1$ and choose the branch of square root in \eqref{e:bdd002} according to $\sqrt{1}=1$. Note that in writing \eqref{e:7.3}, we have implicitly chosen a branch of $u$. This choice will be specified later, within the proof of Theorem \ref{thm:7.1}, and does not matter for the moment.
 
  For the moment, let $\x^{\mathrm{ref}}=(0,0,1)=\xNP$, and let $A_{\x^{\mathrm{ref}}} = A_{\xNP}= I_3$ be the $3\times3$ identity matrix. Consider the corresponding $(\epsilon,\beta)$ coordinate system according to \eqref{e:bdd002}. Then, by \citep[Appendix B]{partAcomplexSphere}, we have
\begin{equation}
    \tilde \Delta_{\c}  = (\epsilon^4 - \epsilon^2) \ptl^2_{\epsilon}
    + \epsilon^3 \ptl_{\epsilon} + \epsilon^2    \ptl^2_{\beta}. \label{e:f009}
\end{equation}
Therefore, using rotational invariance of the (complexified) Laplace-Beltrami operator, we find that the CLBE \eqref{e:f007} becomes, for any $\x^{\mathrm{ref}} \in \mathbb{S}$, rewritten as 
    \begin{align}
    \left[(\epsilon^4 - \epsilon^2) \ptl^2_{\epsilon}
    + \epsilon^3 \ptl_{\epsilon} + \epsilon^2    \ptl^2_{\beta} + \lambda(\lambda+1)\right]u(\epsilon,\beta; \x^{\mathrm{ref}}) =0. \label{e:7.3}
    \end{align}
The equation  \eqref{e:7.3} is separable and singular at $\epsilon =0$. So, we may pose a Frobenius-type ansatz, and look for solutions to \eqref{e:7.3} of the form
    \begin{align}
    u(\beta, \epsilon; \x^{\mathrm{ref}}) = \epsilon^{\nu} \sum_{n = 0}^{\infty}
    F_n(\beta; \x^{\mathrm{ref}}) \epsilon^n, \label{e:AS042}
    \end{align}
    for some unknown $\nu$ to be determined.

Assuming without loss of generality that $F_0\neq0$, substituting \eqref{e:AS042} into \eqref{e:7.3}, and denoting $Q(a) = a(a-1) - \lambda(\lambda + 1)$, we obtain the recursive equation
\begin{equation}
Q(\nu + n +2) F_{n+2} (\beta; \x^{\mathrm{ref}}) 
=
(\nu + n)^2 F_n  (\beta; \x^{\mathrm{ref}})  + \ptl^2_\beta F_n  (\beta; \x^{\mathrm{ref}}) ,
\label{e:AS043}
\end{equation}
and we find that there are two possible values for $\nu$, which we denote by $\nu_{1}$ and $\nu_2$: 
\begin{equation}
\nu_1 = \lambda + 1, 
\qquad
\nu_2 = - \lambda.
\label{e:AS045}
\end{equation}
We also find that $F_{2n+1}(\beta)\equiv0$ for all $n\in\mathbb{N}$.
We thus conclude that \textit{any} solution $u(\z)$ to \eqref{e:7.3}, for $|\epsilon|$ sufficiently small, can be expressed as 
\begin{align*}
u(\beta, \epsilon; \x^{\mathrm{ref}}) = &\phantom{+}\epsilon^{\lambda+1}\left[F_{1,0}(\beta; \x^{\mathrm{ref}}) +\sum_{n=1}^{\infty}F_{1,2n}(\beta;\x^{\mathrm{ref}}) \epsilon^{2n} \right] \\
& +\epsilon^{-\lambda} \left[F_{2,0}(\beta;\x^{\mathrm{ref}}) + \sum_{n=1}^{\infty} F_{2,2n}(\beta; \x^{\mathrm{ref}}) \epsilon^{2n} \right] 
 \numberthis
\label{e:AS046}
\end{align*}
for small $|\epsilon|$. Moreover, because of the recurrence relation \eqref{e:AS043}, all the $F_{j,2n}$ terms are fully determined by $F_{j,0}$, where $j\in\{1,2\}$. This is reminiscent of what is observed in the planar setting, as will be discussed in Section  \ref{sec:analogies}. 
Note, moreover, that depending on the value of $\lambda$, either term may decay, or grow, as $\epsilon \to 0$.

\begin{theorem}\label{thm:7.1}
    Consider $p \in \Xi \setminus (\Phi_+(\mathbb{S}) \cup \Phi_-(\mathbb{S}))$, and let $\x^{\mathrm{ref}} \in \mathbb{S}$. Let
     \begin{align}
             p(\beta) = \left[A_{\x^{\mathrm{ref}}}( \cos(\beta), \sin(\beta), i)\right],  \quad  A(\lambda) &= \frac{2^{- \lambda - 1} \Gamma \left( - \frac{1}{2} -\lambda
		\right)}{\sqrt{\pi} \Gamma (- \lambda)}, \label{e:7.09}
        \end{align}
    and let $\mathcal{D}_1$ and $\mathcal{D}_2$ denote the directivities introduced in \eqref{e:New14} and \eqref{e:New18}.
    Then, we have
        \begin{align}
            F_{1,0}(\beta; \x^{\mathrm{ref}})= \frac{A(\lambda)}{4 \sin(\pi \lambda)}\mathcal{D}_1(p(\beta);\x^{\mathrm{ref}}), \quad F_{2,0}(\beta; \x^{\mathrm{ref}})= \frac{A(-\lambda-1)}{4 \sin(\pi \lambda)}\mathcal{D}_2(p(\beta);\x^{\mathrm{ref}}). \label{e:7.8}
        \end{align} 
     Note that, according to the results of \citep[Section 2]{partAcomplexSphere}, any point $p$ on $\Xi \setminus (\Phi_+(\mathbb{S}) \cup \Phi_-(\mathbb{S}))$ may be parametrised according to  $p= p(\beta)$ for some appropriate $\beta \in \mathbb{C}$, so the chosen parametrisation of $p$ does not lose any generality. 

     Moreover, due to \eqref{e:AS043}, the directivities $\mathcal{D}_1$ and $\mathcal{D}_2$ fully determine all the coefficients $F_{1,2n}$ and $F_{2,2n}$, and thus $u$, in the Expansion \eqref{e:AS046}.
\end{theorem}

 \begin{proof}
       Let 
       \begin{align}
       \bdeta(\beta)= A_{\x^{\mathrm{ref}}}(\cos(\beta), \sin(\beta),i), ~ \text{ so }~p(\beta) = [\bdeta(\beta)], \label{e:7.10}
       \end{align}
      and let $\y \in \partial \mathbb{S}$. As $\epsilon \to 0$, following the arguments detailed in \citep[Section 3(a)]{partAcomplexSphere}, we find that
    \begin{align*}
        G(\z;\y) =&\phantom{+}
	 \frac{A(\lambda) }{4\sin (\pi \lambda)}  
	\epsilon^{1+\lambda} \left[w_1(\y, p(\beta); \x^{\mathrm{ref}})+ \mathcal{O}(\epsilon^2) \right]\\
    & + 
	\frac{A(-\lambda-1) }{4\sin (\pi \lambda)} \epsilon^{-\lambda} \left[w_2(\y, p(\beta); \x^{\mathrm{ref}}) + \mathcal{O}(\epsilon^2)
	\right] \numberthis \label{e:7.9}
    \end{align*}
    where $w_1$ and $w_2$ are the plane-waves given in \eqref{def:planewavesderaph}.

  In \eqref{e:7.9}, we choose the branch of $G(\z;\y)$ as follows. For some fixed $\y=\y^{\star} \in \partial \mathbb{S}$, the expansion \eqref{e:7.9} is valid for the physical branch of $G(\z; \y^{\star})$. 
  The branch for general $\y \in \partial \mathbb{S}$ is then obtained by continuity along $\partial \mathbb{S}$. Now, according to Theorem \ref{thm:6.5}, the Green's third identity
        \begin{align}
            u(\z) = \int_{\partial \mathbb{S}} \G[u(\y), G(\z;\y)] \label{e:7.11}
        \end{align}
    is valid for $\z \in \mathcal{S}^2_{\c} \setminus (T(-\mathbb{S})\cup T(\mathbb{S}))$. In  \eqref{e:7.11}, the branch of $u(\z)$ is determined by the branch of $G(\z;\y)$ according to \eqref{e:6.9}. 
    We use \eqref{e:7.9} within \eqref{e:7.11} to obtain 
        \begin{align*}
            u(\epsilon,\beta) = & \phantom{+} \frac{A(\lambda)}{4 \sin(\pi \lambda)} \epsilon^{1+\lambda} \left[\int_{\partial \mathbb{S}} \G[u(\y), w_1(\y,p(\beta); \x^{\mathrm{ref}})] + \mathcal{O}\left(\epsilon^{2}\right) \right]\\
            &+ \frac{A(-\lambda-1)}{4 \sin(\pi \lambda)} \epsilon^{-\lambda}\left[\int_{\partial \mathbb{S}} \G[u(\y), w_2(\y,p(\beta); \x^{\mathrm{ref}})] + \mathcal{O}\left(\epsilon^{2}\right)\right]. \label{e:7.12} \numberthis
        \end{align*}
      where $\G$ is defined according to \eqref{e:6.24}. 
      Since $\x^{\mathrm{ref}} \in \mathbb{S}$, we may use the definitions \eqref{e:New14} and \eqref{e:New18} of $\mathcal{D}_1$ and $\mathcal{D}_2$ within \eqref{e:7.12}. This gives
        \begin{align*}
            u(\epsilon,\beta; \x^{\mathrm{ref}}) = &\phantom{+}\epsilon^{1+\lambda}  \left[\frac{A(\lambda)}{4 \sin(\pi \lambda)} \mathcal{D}_{1}(p(\beta); \x^{\mathrm{ref}}) 
            + \mathcal{O}\left(\epsilon^{2}\right) \right]\\
            &+ \epsilon^{-\lambda}\left[\frac{A(-\lambda-1)}{4 \sin(\pi \lambda)} \mathcal{D}_{2}(p(\beta); \x^{\mathrm{ref}}) + \mathcal{O}\left(\epsilon^{2}\right)\right]. \label{e:7.13} \numberthis
        \end{align*}
        The equalities given in \eqref{e:7.8} now follow from comparing the $\epsilon^{1+\lambda}$ and $\epsilon^{-\lambda}$ coefficients in  \eqref{e:AS046} with \eqref{e:7.13}. 
    \end{proof}

Due to \eqref{e:7.13}, the point $p \in \Xi$ can be interpreted as an observation direction \textit{relative to the reference point} $\x^{\rm ref}$. Moreover, $p$ \textit{also} serves as a spectral parameter for integration in the plane-wave representations  \eqref{e:New13}, \eqref{e:New014}, and \eqref{e:PWNew023raph}. These observations are reminiscent of the planar setting, as will be discussed in Section \ref{sec:analogies} below.

\begin{remark}\label{rem:7.3}
            Let $\epsilon, \beta \in \mathbb{C}$. Consider the following objects.
            \begin{enumerate}
                \item The loop $\tilde\Gamma_{\!\!\infty}(\z) = \tilde\Gamma_{\!\!\infty}(\epsilon,\beta; \x^{\mathrm{ref}})$, centred at the point $\z$, which is parametrised according to \eqref{e:bdd002};
                \item The loop $\tilde{\Gamma}$ used in the formulation of the real Green's third identity \eqref{e:Greens3rdSphere};
                \item The figure-of-eight loop $\gamma_{\infty}$ shown in Figure \ref{fig:PWApp01}, right.
            \end{enumerate}
          Using the theory of complex projective spaces (cf.\! Remark \ref{rem6.3}), it can be shown that, as $\epsilon \to 0$, $\tilde{\Gamma}_{\!\! \infty}(\epsilon,\beta; \x^{\mathrm{ref}})$ and $\gamma_{\infty}$ are homotopic when considered as loops within $\hat{\mathcal{S}}^2_{\c}$. Here, all homotopies are such that no singularities of corresponding integrands are ever hit, \textit{except for $\Xi$,} where $u$'s and $G$'s asymptotics allow for regularisation of the 1-form $\G[u(\z'); G(\z;\z')]$ as $|\z'| \to \infty$. This asymptotic procedure actually recovers the plane-wave representation \eqref{e:PWNew023}, which is why we depict $\tilde\Gamma_{\!\!\infty}(\z_0)$ as a figure-of-eight loop in Figure \ref{fig:AC02}. In this sense, the plane-wave representation  \eqref{e:PWNew023} may be viewed as `dual' of the Green's third identity \eqref{e:Greens3rdSphere}, whereas the plane-wave representations \eqref{e:New16} and \eqref{e:New17} are, instead, reminiscent of a Fourier-transform. 
    \end{remark}

\section{A summary of analogies between scattering on a plane and scattering on a sphere}\label{sec:analogies}

With this final section, we aim to summarise all the analogies between the planar and the spherical case that have been highlighted throughout the article.    

    \paragraph*{Green's third identities}
Consider the Green's third identities \eqref{e:Greens3rdSphere} and \eqref{e:AS002}. As was pointed out in Section \ref{sec:motivation} these formulae look conceptually quite similar. This similarity is even more prominent when rewriting \eqref{e:Greens3rdSphere} and \eqref{e:AS002} using differential forms, as discussed in Appendix \ref{app:diffSph}, and similar formulae are also valid on Hyperbolic space \citep{Borthwick2016, Chossat2020}. Moreover, let us note that both, the spherical setting and the planar setting admit a complexified version of Green's third identity. Such complexified Green's third identity in the planar setting is discussed in detail in \citep[Appendix B]{AssierShanin2021AnalyticalCont}, which requires analytical continuation of both, the wave-field and Green's function onto the plane's complexification, which is $\mathbb{C}^2$.


\paragraph*{Directivity patterns} 

Consider the directivities $\mathcal{D}_1(p;\x^{\rm ref})$ and $\mathcal{D}_2(p;\x^{\rm ref})$, defined in \eqref{e:New14} and \eqref{e:New18}, and the directivity $\mathcal{D}(\theta)$ introduced in Section \ref{sec:motivation}. These directivities share some important conceptual similarities, which we now highlight.

First, we note that in both cases, the directivities are represented via conceptually similar contour integrals on the physical domain (see \eqref{e:AS008} and \eqref{e:New14}--\eqref{e:New18}). Moreover, they fully determine the far-field behaviour of $u$, just like in the planar setting (see \eqref{e:AS007} and the discussion of Section \ref{sec:farfield}). In both settings, the scattered field $u$ can be represented as a complex contour integral involving the directivities and plane waves (see \eqref{e:AS010} and \eqref{e:New12}--\eqref{e:New13} as well as \eqref{e:PWNew023raph}). We also note that the observation directions ($\theta$ on the plane and $p$ on the sphere) serve as spectral variables for integration. Finally, even though the observation direction $\theta$ in the planar setting seems to be independent of any choice of reference point, the reference point in the planar setting is, by convention, located at the origin of $\mathbb{R}^2$. So, the fact that observation directions are defined relative to a reference point is yet another similarity between the planar and the spherical setting.

Despite these similarities, there are also some fundamental differences between the spherical and planar directivity patterns. On the sphere, two directivities are present whereas there is only a single one in the plane. Moreover, depending on the value of $\lambda$, the scattered field $u$ is not always decaying in the far-field as discussed in Section \ref{sec:farfield}. In the planar case, however, $u$ always decays in the far field, as is evident from \eqref{e:AS007}. Another important difference is the fact that there exists a third plane-wave representation of $u$ that involves \textit{both} directivities \eqref{e:PWNew023raph}. No analogue of this formula exists in the planar setting.

\section{Conclusion and perspectives}
Building on the notion of spherical plane waves introduced in \cite{partAcomplexSphere}, we provided a set of tools and concepts to study and characterise scattering problems on $\S^2$. We have highlighted strong analogies between those, and the classical tools to study scattering problems on the plane. The main such concepts introduced are plane-wave representations of wave fields and directivities. Using Green's third identity we have also explained how wave-fields defined on $\S^2$ can be analytically continued to the complexified sphere $\S_{\c}^2$. This allowed us to formulate a complexified version of Green's third identity, and to make sense of the complexification of the inhomogeneous equation that usually governs the Green's function on $\S^2$. 

Everything we did in this article was done for an arbitrary scatterer, but we noted specific scatterers of interest, such as spherical caps or slits. What we hope to do next is to focus on the slit geometry and show that as in the planar case, the directivities can be thought of a solutions to some ODEs, with the aim of shedding more light on the three-dimensional problem of diffraction by a plane sector.

	\bibliographystyle{unsrt}
	\bibliography{bibliography}

\begin{thebibliography}{10}

\bibitem{Smyshlyaev1993}
V.~P. Smyshlyaev.
\newblock {T}he {H}igh-{F}requency {D}iffraction of {E}lectromagnetic {W}aves
  by {C}ones of {A}rbitrary {C}ross {S}ections.
\newblock {\em SIAM J. Appl. Math.}, 53(3):670--688, 1993.

\bibitem{shanin2005modified}
AV~Shanin.
\newblock Modified smyshlyaev’s formulae for the problem of diffraction of a
  plane wave by an ideal quarter-plane.
\newblock {\em Wave motion}, 41(1):79--93, 2005.

\bibitem{AssierPeake2012}
R.~C. Assier and N.~Peake.
\newblock On the diffraction of acoustic waves by a quarter-plane.
\newblock {\em Wave Motion}, 49(1):64--82, 2012.

\bibitem{sarigiannidis2026integral}
Iraklis Sarigiannidis, Matteo Albani, and Massimiliano Casaletti.
\newblock Integral equation numerical approach for the determination of the
  diffraction coefficients from generally shaped cones.
\newblock {\em COMPEL}, 45(3):412--430, 2026.

\bibitem{AssierPoonPeake2016}
R.~C. Assier, C.~Poon, and N.~Peake.
\newblock Spectral study of the laplace-beltrami operator arising in the
  problem of acoustic wave scattering by a quarter-plane.
\newblock {\em Q. J. Mech. Appl. Math.}, 69(3):281--317, 2016.

\bibitem{partAcomplexSphere}
A.V. Shanin, V~Kunz, and R.C. Assier.
\newblock {Plane-wave representation for the Laplace--Beltrami equation on a
  sphere. Application to the Green's function}.
\newblock {\em Accepted in Proc Roy Soc A; arXiv:2603.24417}, 2026.

\bibitem{AssierShanin2021AnalyticalCont}
R.~C. Assier and A.~V. Shanin.
\newblock Analytical continuation of two-dimensional wave fields.
\newblock {\em Proc. R. Soc. A Math. Phys. Eng. Sci.}, 477(20200681), 2021.

\bibitem{keller1999singularities}
Joseph~B Keller.
\newblock Singularities at the tip of a plane angular sector.
\newblock {\em Journal of Mathematical Physics}, 40(2):1087--1092, 1999.

\bibitem{AssierAbrahams2021}
R.~C. Assier and I.~D. Abrahams.
\newblock A surprising observation on the quarter-plane diffraction problem.
\newblock {\em SIAM J. Appl. Math.}, 81(1):60--90, 2021.

\bibitem{AssierPeake2012FarField}
R.~C. Assier and N.~Peake.
\newblock Precise description of the different far fields encountered in the
  problem of diffraction of acoustic waves by a quarter-plane.
\newblock {\em IMA J. Appl. Math.}, 77(5):605--625, 2012.

\bibitem{AssierShanin2019}
R.~C. Assier and A.~V. Shanin.
\newblock Diffraction by a quarter-plane. {A}nalytical continuation of spectral
  functions.
\newblock {\em Q. J. Mech. Appl. Math.}, 72(1):51--86, 2019.

\bibitem{AssierShanin2021VertexGreensFunctions}
R.~C. Assier and A.~V. Shanin.
\newblock Vertex {G}reen’s functions of a quarter-plane. links between the
  functional equation, additive crossing and {L}am\'e functions.
\newblock {\em Q. J. Mech. Appl. Math.}, 74(3):251--295, 2021.

\bibitem{assier2024contribution}
Rapha{\"e}l~C Assier, Andrey~V Shanin, and Andrey~I Korolkov.
\newblock {A contribution to the mathematical theory of diffraction. Part II:
  Recovering the far-field asymptotics of the quarter-plane problem}.
\newblock {\em Q. J. Mech. Appl. Math.}, 77(1-2):hbae005, 2024.

\bibitem{lyalinov2013scattering}
Mikhail~A Lyalinov.
\newblock Scattering of acoustic waves by a sector.
\newblock {\em Wave Motion}, 50(4):739--762, 2013.

\bibitem{Shanin2005b}
A.~V. Shanin.
\newblock Coordinate equations for a problem on a sphere with a cut associated
  with diffraction by an ideal quarter-plane.
\newblock {\em Quarterly Journal on Mechanics and Applied Mathematics},
  58(2):289--308, 2005.

\bibitem{shanin2012asymptotics}
AV~Shanin.
\newblock Asymptotics of waves diffracted by a cone and diffraction series on a
  sphere.
\newblock {\em Journal of Mathematical Sciences}, 185(4):644--657, 2012.

\bibitem{ValyaevShanin2012}
V.~Valyaev and A.~Shanin.
\newblock {E}mbedding formulae for {L}aplace–{B}eltrami problems on the
  sphere with a cut.
\newblock {\em Wave Motion}, 49:83--92, 2012.

\bibitem{Williams1982}
Marc~H. Williams.
\newblock Diffraction by a finite strip.
\newblock {\em QJMAM}, 35:103--124, 1982.

\bibitem{Shanin2001}
A.~V. Shanin.
\newblock {T}hree {T}heorems {C}oncerning {D}iffraction by a {S}trip or a
  {S}lit.
\newblock {\em Quart. J. Mech. Appl. Math}, 54(1):107--137, 2001.

\bibitem{Shanin2003}
A.~V. Shanin.
\newblock {D}iffraction of a {P}lane {W}ave by {T}wo {I}deal {S}trips.
\newblock {\em Quart. J. Mech. Appl. Math}, 56(2):187--215, 2003.

\bibitem{Babich2007}
V.~M. Babich, M.~A. Lyalinov, and V.~E. Grikurov.
\newblock Diffraction theory: The {S}ommerfeld-{M}alyuzhinets technique (alpha
  science series on wave phenomena).
\newblock {\em Oxford: Alpha Science}, 2007.

\bibitem{Nethercote2020ReviewPaper}
M.~A. Nethercote, R.~C. Assier, and I.~D. Abrahams.
\newblock Analytical methods for perfect wedge diffraction: a review.
\newblock {\em Wave Motion}, 93, 2020.

\bibitem{Latta1956}
G.~E. Latta.
\newblock The solution of a class of integral equations.
\newblock {\em Journal of Rational Mechanics and Analysis}, 5(5):821--834,
  1956.

\bibitem{Bateman1953}
H.~Bateman and A.~Erdelyi.
\newblock {\em Higher Transcendental Functions}, volume~1.
\newblock McGraw-Hill, New York, 1953.

\bibitem{griffiths1984principles}
Phillip~A. Griffiths and Joseph Harris.
\newblock {\em Principles of Algebraic Geometry}.
\newblock Wiley, New York, 1994.

\bibitem{Hatcher2000}
A.~Hatcher.
\newblock {\em Algebraic topology}.
\newblock Cambridge Univ. Press, 2000.

\bibitem{Borthwick2016}
David Borthwick.
\newblock {\em Spectral Theory of Infinite-Area Hyperbolic Surfaces}, volume
  318 of {\em Progress in Mathematics}.
\newblock Birkh{\"a}user, Cham, 2 edition, 2016.

\bibitem{Chossat2020}
Pascal Chossat.
\newblock The hyperbolic model for edge and texture detection in the primary
  visual cortex.
\newblock {\em The Journal of Mathematical Neuroscience}, 10(1):2, 2020.

\bibitem{Shabat1991}
B.~V. Shabat.
\newblock {\em Introduction to Complex Analysis Part II Functions of Several
  Variables}.
\newblock American Mathematical Society, Providence, Rhode Island, 1991.

\bibitem{KobayashiNomizu}
S.~Kobayashi and K.~Nomizu.
\newblock {\em Foundations of Differential Geometry}, volume I $\&$ II.
\newblock Interscience Publishers, John Wiley $\&$ Sons Inc, New York, London,
  Sydney, 1969.

\end{thebibliography}


\appendix

\section{Proof of Proposition \ref{prop:gammainfinityG}} \label{App:New01}

The purpose of this appendix is to prove that the equation  \eqref{e:NewPW23}
  is true. To achieve this, the following lemma will be useful.

\begin{lemma}\label{lem:NewA2}
Let $j\in\{1,2\}$, and let $\x,\y \in \mathcal{S}^2$ with $\x \neq \pm \y$. Then, the equality 
  \begin{equation}
    2 \int_{\gamma} w_j (\boldsymbol{x}, p ; - \boldsymbol{y})
    \Omega_{\boldsymbol{y}} = \int_{\gamma_\infty} w_j (\boldsymbol{x},
    p ; - \boldsymbol{y}) \Omega_{\boldsymbol{y}} + \int_{\hat\gamma_\infty}
    w_j (\boldsymbol{x}, p ; - \boldsymbol{y}) \Omega_{\boldsymbol{y}}
    \label{e:PWApp03}
  \end{equation}
  is true. Here $\gamma_\infty$ is defined as in Proposition \ref{prop:gammainfinityG}. The contour $\hat \gamma_\infty$ is defined similarly: it is a figure-of-eight loop on $\Xi$ encircling $\Phi_+ (\boldsymbol{y})$
  once in the negative direction, $\Phi_- (\boldsymbol{y})$ once in the positive
  direction, and with point $p_0$ common to both, $\gamma_\infty $
  and $\hat \gamma_\infty $, as shown in Figure \ref{fig:PWApp02}, right.
\end{lemma}


\begin{proof}
  This lemma follows from contour deformation on $\Xi$, as illustrated in
  Figure \ref{fig:PWApp02}.
  
  \textbf{Step 1.} Place the two copies of $\gamma$ on $\Xi$.
  
  \textbf{Step 2.} We choose illustration on a stereographic projection of
  $\Xi$. On the left, the two loops have been dragged on $\Xi$ such that they
  encircle the branch cuts closely on one side, and such that they are joined
  at the point $p_0$. Thereafter, the two copies of $\gamma$ have been
  concatenated at $p_0$ to form the figure-of-eight loop $\gamma_8$ shown on
  the right.
  
  \textbf{Step 3.} We drag the segments of contours that lie close to the
  branch cuts through the cuts, as indicated by the black arrows in Step 2,
  right. After concatenating the parts that have been dragged through the
  branch cuts akin to how we concatenated $\gamma + \gamma$ to form
  $\gamma_\infty$, we will have reached the configuration $\gamma_\infty
  + \hat\gamma_\infty$. This gives Equation \eqref{e:PWApp03}.
\end{proof}

\begin{figure}[h!]
    \centering
    \includegraphics[width=\linewidth]{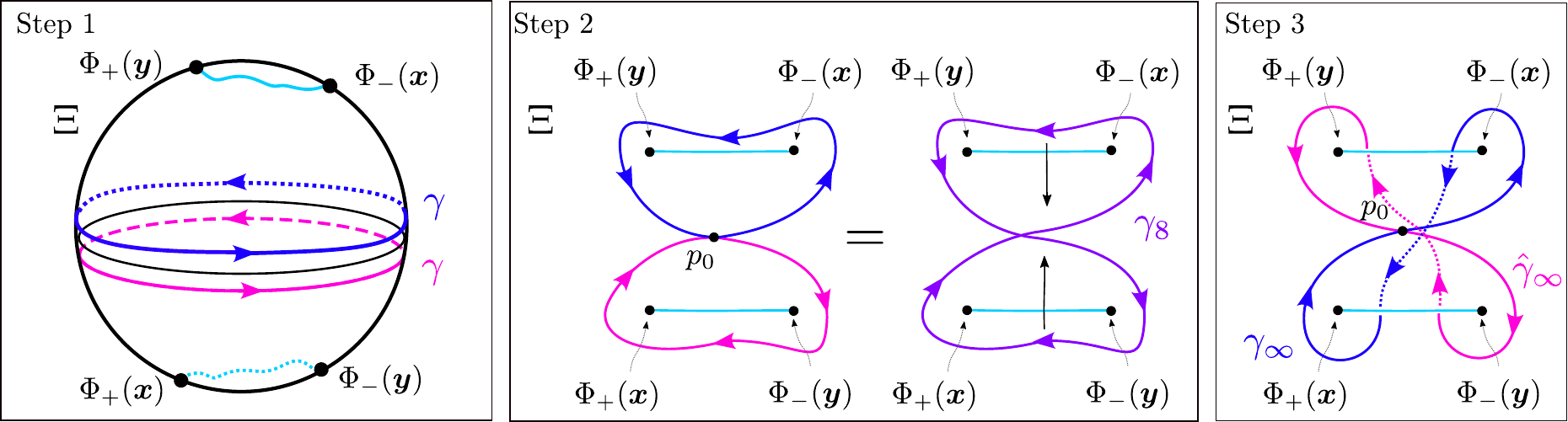}
    \caption{\small Steps used in the proof of Lemma \ref{lem:NewA2}}
    \label{fig:PWApp02}
\end{figure}

Now, within \eqref{e:PWApp03}, consider the integral
    $
    \int_{\hat \gamma_\infty } w_j (\boldsymbol{x}, p; - \boldsymbol{y}) \Omega_{\boldsymbol{y}}.
    $
We denote by $\Upsilon  : \Xi \rightarrow \Xi$ the M\"obius transformation uniquely determined by requiring  that it
maps $\Phi_{\mp} (\boldsymbol{y}) \mapsto \Phi_{\pm} (\boldsymbol{x})$ and
$\Phi_{\pm} (\boldsymbol{x}) \mapsto \Phi_{\mp} (\boldsymbol{y})$, similar to what was done 
in
\citep[Appendix D]{partAcomplexSphere}. 
As in \citep{partAcomplexSphere}, we have
\begin{equation}
  \Upsilon^{\star}  (w_j (\boldsymbol{x}, p ; - \boldsymbol{y})
  \Omega_{\boldsymbol{y}}) = w_j (\boldsymbol{y}, p ; - \boldsymbol{x})
  \Omega_{\boldsymbol{x}}
\end{equation}
and it can be checked that $\Upsilon (\hat \gamma_\infty ) =
\gamma_\infty $. 
Thus, akin to the computations carried out in \citep[Appendix D]{partAcomplexSphere},  we find that the following Lemma is
true.

\begin{lemma}
  Let $j\in\{1,2\}$, and let $\x,\y \in \mathcal{S}^2$ with $\x \neq \pm \y$. Then, the equality
  \begin{equation}
    2 \int_{\gamma} w_j (\boldsymbol{x}, p ; - \boldsymbol{y})
    \Omega_{\boldsymbol{y}} = \int_{\gamma_\infty } w_j (\boldsymbol{x},
    p ; - \boldsymbol{y}) \Omega_{\boldsymbol{y}} + \int_{\gamma_\infty }
    w_j (\boldsymbol{y}, p ; - \boldsymbol{x}) \Omega_{\boldsymbol{x}}
    \label{e:PWApp03new}
  \end{equation}
  is true.
  \label{lem:lemA2}
\end{lemma}
We will need two more results. First \citep[Eq. (D.13)]{partAcomplexSphere}:
\begin{align}
w_2(\y,p; -\x) &= e^{- i \pi (\lambda + 1)/2} w_2(\y,p;-\x^{\mathrm{ref}}) w_1(\x,p; \x^{\mathrm{ref}}) \frac{\x\cdot\bdeta}{\x^{\rm ref}\cdot\bdeta},
\label{eq:D13fromPartA}
\end{align}
valid for any $\x$, $\y$ and $\x^{\rm ref}$ in $\S^2$, where $\bdeta$ is such that $[\bdeta]=p$ and $\bdeta\cdot\x^{\rm ref}\neq0$,  i.e.\ $p \neq \Phi_{\pm}(\x^{\mathrm{ref}})$. Second we will need \citep[ Lemma D.3]{partAcomplexSphere}, stating that for any $\x$ and $\y$ in $\S^2$, we have
\begin{equation}
    \left( \frac{\x \cdot \boldsymbol{\eta}}{\y \cdot \boldsymbol{\eta}} \right) \Omega_{\boldsymbol{x}} =
    \Omega_{\boldsymbol{y}}, \label{e:New09}
  \end{equation}
  provided that $\y\cdot\bdeta\neq0$, i.e.\ $p \neq \Phi_{\pm}(\y)$
Let us start by giving a simple expression for $G$. Letting $\x^{\rm ref}\to\y$ in \eqref{e:New01}, and using the normalisation \eqref{eq:planewavesbranch}, we directly find that
\begin{equation}
 G (\boldsymbol{x} ; \boldsymbol{y}) = \mathcal{B}_{\lambda} e^{i \pi
    (\lambda + 1) / 2} \int_{\gamma} w_2
  (\boldsymbol{x}, p ; - \boldsymbol{y}) \Omega_{\boldsymbol{y}}.
  \label{eq:raphappA9}
  \end{equation}
  Therefore, using Lemma \ref{lem:lemA2} for $j=2$, we obtain directly 
\begin{equation}
G (\boldsymbol{x} ; \boldsymbol{y})=\frac{\mathcal{B}_{\lambda}}{2} \left ( \int_{\gamma_\infty } e^{i \pi
    (\lambda + 1) / 2} w_2 (\boldsymbol{x},
    p ; - \boldsymbol{y}) \Omega_{\boldsymbol{y}} + \int_{\gamma_\infty }
    e^{i \pi
    (\lambda + 1) / 2} w_2 (\boldsymbol{y}, p ; - \boldsymbol{x}) \Omega_{\boldsymbol{x}}\right).
    \label{eq:raph-inter-fast}
\end{equation}
Multiplying  \eqref{eq:D13fromPartA} by  $e^{i \pi
    (\lambda + 1) / 2}\Omega_{\x}$ and using \eqref{e:New09}, we get
    \begin{equation}
        e^{i \pi
    (\lambda + 1) / 2} w_2 (\boldsymbol{y}, p ; - \boldsymbol{x}) \Omega_{\boldsymbol{x}}= w_2(\y,p;-\x^{\mathrm{ref}}) w_1(\x,p; \x^{\mathrm{ref}})\Omega_{\boldsymbol{x}^{\rm ref}}
    \label{eq:fast-step1}
    \end{equation}
Now swapping $\x$ and $\y$ in \eqref{eq:fast-step1}, we get
    \begin{equation}
        e^{i \pi
    (\lambda + 1) / 2} w_2 (\boldsymbol{x},
    p ; - \boldsymbol{y}) \Omega_{\boldsymbol{y}}=w_2(\x,p;-\x^{\mathrm{ref}}) w_1(\y,p; \x^{\mathrm{ref}})\Omega_{\boldsymbol{x}^{\rm ref}}.
    \label{eq:fast-step2}
    \end{equation}
Using \eqref{eq:fast-step1} and \eqref{eq:fast-step2} within \eqref{eq:raph-inter-fast}, and using the fact that 
\begin{align}
	w_2(\y,p;-\x^{\rm ref})w_1(\x,p;\x^{\rm ref})\Omega_{\x^{\rm ref}} =  - w_2(\y,p;\x^{\rm ref})w_1(\x,p;- \x^{ \rm ref})\Omega_{\x^{\rm ref}}, 
\end{align}
which follows from \eqref{eq:D13fromPartA} upon substituting $\x^{\rm ref} \to -\x^{\rm ref}$,   
directly gives the sought-after representation \eqref{e:NewPW23}, which concludes the proof. \hfill $\Box$ 
\begin{remark}
    We could instead obtain an expression akin to \eqref{eq:raphappA9} by starting from \eqref{e:New02}, and then use $j=1$ in Lemma \ref{lem:lemA2}. This would lead to the same result.
\end{remark}

\section{On a differential form formalism}\label{App:Diff}

Throughout this appendix, we will recast our results in terms of the differential form formalism. For an introduction to the corresponding mathematical theory, see \citep{Shabat1991} or \citep{KobayashiNomizu}. We will, in particular, highlight similarities between planar and spherical wave scattering. 

\subsection{Scattering on the plane}\label{sec:PlaneDiff}

Let us now focus on wave scattering on the sphere $\S^2$. This section is focused on the planar setting. Consider Green's third identity \eqref{e:AS002}. As was discussed in detail in \citep[Appendix B]{AssierShanin2021AnalyticalCont}, \eqref{e:AS002} can be rewritten by using  Cartan's  differential form formalism. To that end, let us denote by $\star$ the Hodge-star operator, and let $d$ denote the exterior derivative, so that for any smooth function $v$, we have
\begin{equation}
\star \! d v =  \ptl_{x_2} v \, dx_1 - \ptl_{x_1} v \, dx_2, ~ \text{ and } ~ (-\star d \star d) v = \Delta v. 
\label{e:AS004}
\end{equation}
Using the definitions of $\star$, $d$, $\partial_{\n}$ and $d\ell$, it is a straightforward computation to verify the following identity. For any two smooth functions $u$ and $v$, we have
\begin{align}
    \left[ 
 v(\y) \ptl_{\n} u(\y)
-
u (\y) \ptl_{\n} v(\y)
\right] d\ell = v(\y) \left(\star du(\y)\right) - u(\y) \left(\star d v(\y)\right).\label{e:AS?}
\end{align}
In particular, letting $v(\y) = G(\x;\y)$  denote the planar Green's function given in \eqref{e:Greal}, we find that Green's third identity becomes rewritten as
\begin{equation}
u(\x)  = \int_{\tilde{\Gamma}} \left[G(\x;\y) \left(\star du(\y)\right) - u(\y) \left(\star d G(\x;\y)\right)\right]
\label{e:AS006}
\end{equation}
where the operator $\star d$ acts on the $\y$ component. As was also shown in \citep{AssierShanin2021AnalyticalCont}, the form $v(\y) \left(\star du(\y)\right) - u(\y) \left(\star d v(\y)\right)$ is \textit{closed} whenever $u$ and $v$ solve the Helmholtz equation on the plane. So, one may freely deform the contour of integration in \eqref{e:AS006}. Moreover, as detailed in \citep{AssierShanin2021AnalyticalCont} we may \textit{complexify} Green's third identity \eqref{e:AS006} using complex differential forms, and the same strategy works on the sphere, as we now discuss.  

\subsection{Scattering on the sphere}\label{app:diffSph}
Consider, again, the Hodge star operator $\star$ which is now defined relative to the metric $g = \text{diag}(1,\sin^2(\theta))$ (as opposed to the Euclidean $g=\text{diag}(1,1)$ used within the previous Section \ref{sec:PlaneDiff}) and the exterior derivative $d$, which does not depend on the metric. We find that 
	\begin{align}
		 \star d f = \sin(\theta) \partial_{\theta} f d\varphi - \frac{1}{\sin(\theta)}\partial_{\varphi} f d\theta, \quad \text{ and } \quad  (- \star d \star d) f=\tilde{\Delta} f,
	\end{align}
and the left-hand-sides of these equations are defined independently of the coordinate system. Now, complexify the metric $g$ 
to 
the `metric' $g_{\c}$ which, in local \textit{complexified} spherical coordinates, is given by $g_{\c} = \text{diag}(1,\sin^2(\theta))$. Using $g_{\c}$, we can define the corresponding \textit{complexified} Hodge-star operator $\star_{\c}$ taking a $(p,0)$-form to a $(2-p,0)$-form. Let $\partial$ denote the \textit{holomorphic} exterior derivative, such that $d= \partial + \bar{\partial}$ for the Dolbeault operator $\bar\partial$. We obtain 
	\begin{align}
		 \star_{\c} \partial f  =  \sin(\theta) \partial_{\theta} f d\varphi - \frac{1}{\sin(\theta)}\partial_{\varphi} f d\theta ~ \text{ and } ~ -(\star_{\c} \partial \star_{\c} \partial) f = \tilde{\Delta}_{\c} f \label{AppA.2}
	\end{align} 
where $\theta$ and $\varphi$ denote complexified spherical coordinates and $\tilde{\Delta}_{\c} f$ denotes the \textit{complexified} Laplace-Beltrami operator.

We can now  prove the following result, required for the proof of Lemma \ref{lem:GreensComplex}.
	\begin{lemma}\label{Lem:Greens02}
		Consider two functions $f$ and $g$, which are holomorphic and satisfy the CLBE \eqref{e:f007} in  an open domain $U \subset \S^2_{\c}$. The holomorphic $(1,0)$-form 
            \begin{align}
                \G[g,f]=g(\star_{\c}\partial f) - f (\star_{\c}\partial g)  \label{e:A.5}
            \end{align}
        is closed in $U$.
	\end{lemma}

\begin{remark}
	To use the preceding Lemma for the proof of Lemma \ref{lem:GreensComplex}, we choose $g=G$ and $f=u$. 
\end{remark}

 
 \begin{proof}[Proof of Lemma \ref{Lem:Greens02}]
 	 By analyticity of $u$ and $g$, we find $\bar{\partial}(g(\star_{\c}\partial f) - f (\star_{\c}\partial g)) =0$, so $d(g(\star_{\c}\partial f) - f (\star_{\c}\partial g)) =  \partial (g(\star_{\c}\partial f) - f (\star_{\c}\partial g))$. But 
 	 	\begin{align}
 	 		\partial (g (\star_{\c}\partial f) - f (\star_{\c}\partial g)) = \star_{\c}\left((\tilde{\Delta}_{\c}g)f - (\tilde{\Delta}_{\c} f) g\right) =0, \label{eq.AppendixGreens01}
 	 	\end{align}
 	 where we have used $\tilde{\Delta}_{\c} f = - \lambda(\lambda +1)f $ and $\tilde{\Delta}_{\c} g = - \lambda(\lambda +1)g$. Consequently, 
     \begin{align}
     d(g(\star_{\c}\partial f) - f (\star_{\c}\partial g)) =0,
     \end{align} 
     which concludes the proof of the Lemma.
 \end{proof}


We may thus re-write Green's third identity on  $\mathcal{S}^2$, as
            \begin{align}
                u(\x) = \int_{\tilde{\Gamma}}\left[G(\x;\y)(\star d u(\y)) - u(\y)(\star d G(\x;\y))\right] \label{e:A.9}
            \end{align}
which matches the corresponding planar formula \eqref{e:AS006}, 

 The identity \eqref{e:A.9} and its complexified analogue are also valid more generally on other analytic surfaces, like hyperbolic space, wheare the Hodge-star operator takes the negative curvature into account. In hyperbolic space, the Green's function is expressed through the Legendre function of the \textit{second} kind \citep{Borthwick2016}. 

\end{document}